\documentclass[final,3p,times]{elsarticle}

\usepackage{amssymb}

\usepackage{lipsum}
\usepackage{amsfonts}
\usepackage{float}
\usepackage{graphicx}
\usepackage{xcolor}
\usepackage{subfig}
\usepackage[utf8]{inputenc}
\usepackage{epstopdf}
\usepackage{nicefrac}
\usepackage{bm} 
\usepackage{algorithmic}
\ifpdf
  \DeclareGraphicsExtensions{.eps,.pdf,.png,.jpg}
\else
  \DeclareGraphicsExtensions{.eps}
\fi

\usepackage{pifont,latexsym,ifthen,calc}
\usepackage{amsthm,amssymb,amsbsy,amsmath}
\usepackage[errorshow]{tracefnt}
\usepackage{hyperref}
\usepackage{enumerate}
\def\abs#1{\left|{#1}\right|}
\newtheorem{theorem}{Theorem}
\newtheorem{remark}{Remark}
\newtheorem{lemma}{Lemma}
\newtheorem{proposition}{Proposition}
\newtheorem{definition}{Definition}

\journal{Journal of Computational Physics}

\begin{document}

\begin{frontmatter}



\title{A static memory sparse spectral method for time-fractional PDEs}


\author[inst1]{Timon S. Gutleb}
\affiliation[inst1]{organization={Mathematical Institute, University of Oxford},
            city={Oxford},
            postcode={OX2 6GG}, 
            state={England},
            country={United Kingdom}}

\author[inst1]{Jos\'e A. Carrillo}

\begin{abstract}
We discuss a method which provides accurate numerical solutions to fractional-in-time partial differential equations posed on $[0,T] \times \Omega$ with $\Omega \subset \mathbb{R}^d$ without the excessive memory requirements associated with the nonlocal fractional derivative operator. Our approach combines recent advances in the development and utilization of multivariate sparse spectral methods as well as fast methods for the computation of Gauss quadrature nodes with recursive non-classical methods for the Caputo fractional derivative of general fractional order $\alpha > 0$. An attractive feature of the method is that it has minimal theoretical overhead when using it on any domain $\Omega \subset \mathbb{R}^d$ on which an orthogonal polynomial basis is available.
We discuss the memory requirements of the method, present several numerical experiments demonstrating the method's performance in solving time-fractional PDEs on intervals, triangles and disks and derive error bounds which suggest sensible convergence strategies. As an important model problem for this approach we consider a type of wave equation with time-fractional dampening related to acoustic waves in viscoelastic media with applications in the physics of medical ultrasound and outline future research steps required to use such methods for the reverse problem of image reconstruction from sensor data.
\end{abstract}


\end{frontmatter}



\section{Introduction}
The Caputo derivative \cite{caputo1967linear} of order $\alpha >0, \alpha \notin \mathbb{N}$
\begin{align}\label{eq:Caputodefinition}
\frac{\partial^\alpha}{\partial t^\alpha} f(t) = D_C^\alpha f(t)=\frac{1}{\Gamma(\left\lceil \alpha \right\rceil-\alpha)} \int_0^t (t-s)^{(\left\lceil \alpha \right\rceil-\alpha-1)}f^{( \left\lceil \alpha \right\rceil)}(s)\, \mathrm{d}s,
\end{align} 
is a fractional generalization of the conventional integer order derivative
and appears in many mathematical models in the natural science and medical applications \cite{treeby2010modeling,treeby2014modeling,magin2004fractional,goulart2017fractional,dikmen2005modeling,taylor2002kelvin,cesarone2004memory,bounaim2008computations}. There are multiple non-equivalent generalizations of the fractional derivative which have received the attention of mathematical research. In this paper, we focus on Caputo type (as opposed to e.g. Riemann-Liouville type) fractional derivatives exclusively.\\
In the general case of non-integer $\alpha$ the Caputo derivative is nonlocal and computing numerical solutions to fractional differential equations frequently more closely resembles integral rather than differential equation solvers. Treated naively, the nonlocality forces us to use an approach in which the history of the function at each point in time $t$ must be stored -- in principle indefinitely. Since fractional differential equations appear frequently in models used in the natural sciences (some of which we will mention below) where the desired spacetime dimension is usually $3+1$ there has been a great deal of interest in the literature in mitigating the memory accumulation problem inherent to time-fractional differential equations.\\
In this paper we introduce a nonclassical recursive and static memory (`history-free') method for time-fractional partial differential equations on any domain on which sparse spectral methods for classical PDEs are available, with minimal theoretical overhead when adapting it to new problems and domains.\\
A full discussion of numerical methods available for the treatment of fractional PDEs is impossible to give in the context of a single paper and entire monographs have been written on the topic, see e.g. \cite{jin2023numerical,harker2020fractional,baleanu2012fractional}. We nevertheless aim to provide a brief overview of popular approaches for time-fractional PDEs and in particular aim to mention instances where spectral methods were used in such a context: Classical finite difference time-stepping approaches for time-fractional differential equations end up with $O(N^2)$ computational complexity and $O(N)$ memory requirements for $N$ time steps, since as mentioned they effectively require remembering the solution at each previous time-step \cite{jin2023numerical}. Coupling higher dimensional spatial discretizations with such methods thus quickly becomes computationally untenable as the complexity is multiplied by the number of spatial grid points. The computational complexity may be improved by using a sum of exponentials (SOE) approach as described e.g. in \cite[Section 4.2]{jin2023numerical}, which when used in conjunction with appropriately chosen memory cutoffs $N_\ell$ can reduce the memory footprint to $O(N_\ell)$ and the computational complexity to $O(N N_\ell)$. Somewhat simpler approaches which also rely on what has been termed the "short memory principle" were discussed in great detail in \cite{podlubny1998fractional}. A classical reference for a critique as well as proposed improvements for the short memory principle in the form of a \emph{logarithmic} memory principle can be found in \cite{ford2001numerical}. We skip over a contemporary approach to the above references, the nonclassical method due to Yuan and Agrawal  \cite{yuan2002numerical}, as it is the topic of section \ref{sec:previouswork}.\\
Spectral methods for fractional calculus problems of various kinds have recently been successfully used for a number of different problems: A sparse spectral method using a non-orthogonal sum space basis for ODEs involving Caputo or Riemann--Liouville fractional derivatives was introduced by Hale and Olver in \cite{hale2018fast}. More recently, a spectral method for such ODEs was proposed by Pu and Fascondini in \cite{pu2022numerical} which improved upon the stability of Olver and Hale's method by relying on specially constructed orthogonal functions at the cost of no longer resulting in sparse systems. Despite their efficiency and accuracy, neither of these spectral methods allow memory free computations for time-fractional PDEs, which is a potentially critical limitation for real world applications in higher dimensions. A relevant model of such an equation which we consider in some detail in this paper is a lossy wave equation model for acoustic waves in viscoelastic media as e.g. discussed in \cite{treeby2014modeling,treeby2010modeling} for the purposes of developing medical ultrasound treatment devices. Colbrook and Ayton \cite{colbrook2022contour} recently introduced a competitive contour method for time-fractional PDEs based on taking inverse Laplace transforms. Among the advantages of Colbrook and Ayton's method are tight error control and excellent convergence properties. As a drawback one must be able to make a suitable choice of contour which requires knowledge about the singularities of the involved Laplace transform.\\
The primary motivation of this paper is the development of efficient and accurate solvers for the above-mentioned viscoelastic fractional wave equation \cite{treeby2014modeling} (see also \cite{treeby2010modeling}) with applications in medical ultrasound treatment devices:
\begin{align*}
\frac{1}{c_0^2} \frac{\partial^2}{\partial t^2}f - \Delta f - \tau \frac{\partial ^\alpha}{\partial t^\alpha} \Delta f = 0.
\end{align*}
The main domain of interest for this problem is a union of three-dimensional ball and stacked spherical shells of varying thickness as an initial model for the human skull and brain. State-of-the-art methods for this problem, cf. \cite{treeby2014modeling,chen2004fractional}, replace the time-fractional derivative with a spatially nonlocal fractional Laplacian $(-\Delta)^\alpha$ \cite{kwasnicki2017ten,di2012hitchhikers}. While this results in a history-free method, it also introduces an inherent systematic approximation error prior to any considerations of error in the numerical solver. While the full multi-domain model presents additional difficulties which we intend to address in a follow-up paper, we discuss a closely related stand-in on the two-dimensional disk in the numerical experiments in Section \ref{sec:treebynumexp} to demonstrate our method's applicability for this class of problems.\\
This paper is structured as follows: In Section \ref{sec:spectralintro} we give a brief introduction into the theory of and recent advances made in sparse spectral methods. Section \ref{sec:previouswork} provides a brief history of nonclassical methods for Caputo derivatives. Section \ref{sec:recursivesection} shows how to construct recursive history-free sparse spectral methods for time-fractional Caputo PDEs. In Section \ref{sec:memorycost} we consider the question of how much stored data the introduced static memory method requires. Finally, we give an error analysis of the method in Section \ref{sec:error} and present five detailed numerical experiments in Secton \ref{sec:numexpsec} which discuss in practice properties such as error, stability as well as applicability in higher dimensional domains. We conclude with an outlook on future research directions.
\section{Sparse spectral methods}\label{sec:spectralintro}
In this section we introduce and discuss the building blocks of a class of numerical algorithms for solving differential and integral equations called sparse spectral methods. The specific type of spectral method we describe in this paper belongs to the class of \emph{ultraspherical} spectral methods as described in a 2013 paper by Olver and Townsend \cite{olver2013fast}, where the `ultraspherical' is a reference to the ultraspherical polynomials (aka \emph{Gegenbauer} polynomials) used therein. A more recent thorough treatment of these methods can be found in \cite{olver2020fast}. The basic idea behind these methods is to expand functions in bases of orthogonal polynomials in such a way that the linear operators appearing in the problem, e.g. differentiation and integration, have banded matrix forms.\\
We call a complete set of degree-ordered polynomials $\{ p_j(x) \}_{j \in \mathbb{N}_0}$ $orthogonal$ with respect to a positive Borel measure $\mu(x)$ if the polynomials satisfy
\begin{align*}
\langle p_n, p_m \rangle_\mu = \int_\mathbb{R} p_n(x) p_m(x) \mathrm{d}\mu(x) = c_{nm} \delta_{nm},
\end{align*}
where $\delta_{nm}$ is the Kronecker delta and $c_{nm}$ are constants independent of $x$. If  $c_{nm} = 1$ for all $n,m$ then the polynomials are further referred to as \emph{orthonormal}. Throughout this paper we adopt the notation $$\mathbf{P}(x) = \left( p_0(x) \quad p_1(x)\quad p_2(x)\quad \cdots \right).$$ The advantages of this notation is that polynomial expansion of a sufficiently well-behaved function in an orthogonal polynomial basis can be written concisely as
\begin{align*}
f(x) = \mathbf{P}(x) \bm{f} = \sum_{k=0}^\infty f_k p_k(x).
\end{align*}
In a computing context these function approximations are truncated at finite degree $K$. In addition to the above characterization, it is natural to speak of orthogonal polynomials as living on a characteristic domain $\Omega = \rm{supp}{(\mu)}$ determined by the support of the measure. For our purposes the measure $\mu(x)$ will always have an associated density function such that we have $\mathrm{d}\mu(x) = w(x) \mathrm{d}x$ where $\mathrm{d}x$ is the Lebesgue measure and we call $w(x)$ the weight function or simply the weight associated with the orthogonal polynomials. Classical examples of orthogonal polynomials are the well-known Legendre, Chebyshev, Gegenbauer, Jacobi, Laguerre and Hermite polynomals. Of particular importance for spectral methods are the Jacobi polynomials $\mathbf{P}^{(a,b)}(x)$, a two-parameter family of polynomials (with Legendre, Chebyshev, Gegenbauer as special cases) which are orthogonal with respect to the weight function $w(x) = (1-x)^a(1+x)^b$ with $a,b>-1$ on $\Omega = [-1,1]$.\\
Finally, we comment on how to solve integro-differential equations using sparse spectral methods. First, we note that differentiation
\begin{align*}
\frac{\mathrm{d}}{\mathrm{d}x} f(x) = \mathbf{P}(x) \mathcal{D} \bm{f},
\end{align*}
as well as integration are known to satisfy recurrences and are thus banded operators when mapping between appropriate polynomial spaces, see e.g. the explicit forms given in \cite[18.9]{nist_2022}. As an explicit example we consider the entries of $\mathcal{D}$ for Jacobi polynomials and observe that the banded first derivative acts a shift operator between bases as follows \cite[18.9.15]{nist_2022}:
\begin{align*}
\frac{\mathrm{d}}{\mathrm{d}x}P^{(a,b)}_{n}\left(x\right)=\frac{(n+a+b+1)}{2}P^{(a+1,b+1)}_{n-1}\left(x\right).
\end{align*}
Note from the above that obtaining banded derivative operators requires careful choice of target basis (it will generally not be the same as the original basis) and may thus necessitate the use of banded conversion operators \cite[18.9]{nist_2022} elsewhere in the equation where no derivatives operators are present. We will see examples of how this is done in practice in the numerical experiments of this paper. To solve differential equations one usually also requires a way to enforce boundary conditions. In the context of sparse spectral methods this is usually achieved by appending point-evaluation functionals, which take the form of row vectors acting on coefficients, to the linear operator and the corresponding boundary values to the right-hand side, cf. e.g. \cite{olver2013fast,gutleb2021fast}. Alternatively one can also build the boundary conditions directly into the basis, cf. \cite{olver_sparse_2019}. \\
In Section \ref{sec:numexpsec} we describe numerical experiments of our method on triangle and disk domains. We thus note that the discussions in this section can be straightforwardly generalized to higher dimensional domains by using multivariate orthogonal polynomials $\mathbf{P}(\mathbf{x})$, cf. \cite{olver_sparse_2019}. In general higher dimensional domains each variable $x$, $y$ etc. has its own associated block-banded multiplication operator $X$, $Y$ and so on and PDEs with non-zero boundary conditions may be solved using restriction operators analogously to the use of point evaluation operators mentioned above.\\
We close this section with an overview of recent progress on these banded spectral methods on various domains of interest: Since the seminal report on ultraspherical spectral methods in \cite{olver2013fast} a flurry of papers have been published to expand the range of problems that they can be applied to by deriving recurrence relationships which result in banded operators for various types of problems, see e.g. \cite{slevinsky2017fast,gutleb2022sparse,gutleb2020sparse,gutleb2022balls,papadopoulos2022sparse,hale2019ultraspherical,gutleb2022computing}. A directly related line of recent research has concerned the construction of efficiently computable orthogonal polynomial bases on various multivariate domains for the purpose of using them in such sparse spectral methods, recent successful examples of which include disks \cite{vasil_tensor_2016}, triangles and simplices \cite{olver_sparse_2019,olver_recurrence_2018,aktacs2020new}, disk slices and trapeziums \cite{snowball_sparse_2020}, spherical caps \cite{snowball2021sparse}, wedges \cite{olver_orthogonal_2019}, surfaces of revolution \cite{olver_revolution_2020} as well as quadratic and cubic curves \cite{fasondini_orthogonal_2020,olver_orthogonal_2021}. For a recent reference on multivariate orthogonal polynomials on classical higher dimensional domains such as balls, simplices and cubes we refer to the monograph on the subject by Dunkl and Xu \cite{dunkl_orthogonal_2014}. We refer to \cite{gutleb_polynomial_2023} for a detailed discussion of computational aspects relating to constructing orthogonal polynomials with respect to non-classical weight functions, which can be used to construct orthogonal polynomials on non-classical domains such as annuli and spherical bands.\\
Finally, we mention as part of our motivation for using sparse spectral methods the possibility of adapting such methods in the form of spectral element methods which instead of being posed on a single domain work by deconstructing a domain into a mesh of simpler form -- usually by means of triangles as done in \cite{olver_sparse_2019}. For the motivating equation of interest, the natural domain decomposition is to split the sphere model of the skull into ball and spherical shell domains along the discontinuities. Importantly, the option of using spectral element methods with minimal theoretical modifications means that one can alleviate some of the usual criticisms levied against sparse spectral methods in requiring significant amounts of regularity from both the domain and the solution since a function then only needs to be well approximated by piecewise polynomials on a mesh, allowing the use of discontinuities as well as mesh boundary conditions.
\section{Previous work on non-classical methods for Caputo derivatives}\label{sec:previouswork}
The method we describe in this paper belongs to a family of non-classical numerical algorithms for time-fractional differential equations whose origin can be traced back to a 2002 paper by Yuan and Agrawal \cite{yuan2002numerical}. In this section we provide an overview of the development of these methods without claiming historical completeness. Put concisely, Yuan and Agrawal made use of the following result:
\begin{theorem}[Generalized Yuan--Agrawal--Caputo derivative]\label{thm:generalyuanagrawal}
Let $\alpha > 0$, $\alpha \notin \mathbb{N}$ and $f \in C^{\lceil \alpha \rceil}[0,T]$. Then the Caputo fractional derivative of $f$  can be expressed as
\begin{align*}
\frac{\partial^\alpha}{\partial t ^\alpha} f(t) = \int_0^\infty \phi_f(w,t) \mathrm{d}w,
\end{align*}
where the function $\phi_f : (0,\infty) \times [0,T] \rightarrow \mathbb{R}$ is defined by
\begin{align*}
\phi_f(w,t) := \frac{(-1)^{\lfloor \alpha \rfloor}2 \sin(\pi \alpha)}{\pi} w^{2\alpha-2\lceil \alpha \rceil +1} \int_0^t e^{-w^2(t-\tau)} \frac{\partial^{\lceil \alpha \rceil}}{\partial \tau^{\lceil \alpha \rceil}} f(\tau)\mathrm{d}\tau.
\end{align*}
Furthermore, for fixed $w>0$ the function $\phi_f(w,t)$ satisfies the (non-fractional) differential equation
\begin{align}\label{eq:diffeqphiYA}
\frac{\partial}{\partial t} \phi_f(w,t) = -w^2 \phi_f(w,t) + \frac{(-1)^{\lfloor \alpha \rfloor}2 \sin(\pi \alpha)}{\pi} w^{2\alpha-2\lceil \alpha \rceil +1} \frac{\partial^{\lceil \alpha \rceil}}{\partial t^{\lceil \alpha \rceil}} f(t),
\end{align}
with initial condition $\phi(w,0) = 0$.
\end{theorem}
Strictly speaking Yuan and Agrawal themselves only proved this result for $0 < \alpha < 1$, with an extension for $1 < \alpha < 2$ being provided by Trinks and Ruge \cite{trinks2002treatment} in the same year. The fully general version we reproduced above was proved by Diethelm in 2008 \cite{diethelm2008investigation}. The Yuan--Agrawal method for solving equations involving the Caputo fractional derivative computes $\phi_f(w,t)$ via the differential equation \eqref{eq:diffeqphiYA} and then uses Gauss--Laguerre quadrature for the half-infinite integral. As all of the resulting terms then only depend on the functions $f$ and $\phi_f$ at the previous time step this eliminates the need to store the past of $f$ in memory in exchange for the additional computational effort of computing $\phi_f$ and its half-infinite integral. A similar way of rewriting the Caputo fractional derivative using a slightly different integrand function $\phi_*(w,t)$ was introduced by Chatterjee \cite{chatterjee2005statistical} and also later analyzed by Diethelm \cite{diethelm2008investigation} but we will not further discuss this alternative in this paper.\\
The Yuan--Agrawal method initially received mixed responses in the literature due to its poor convergence properties \cite{schmidt2006critique, lu2005wave}, specifically in relation to the number of quadrature points required for the Gauss--Laguerre step, but also resulted in a number of papers being written with the specific intention of improving it. Lu and Hanyga split the half-infinite integral into a region of $[0,1]$ and $[1,\infty)$ and used Gauss--Jacobi and Gauss--Laguerre respectively to improve the convergence behavior of the method \cite{hanyga2005wave,lu2004numerical,lu2005wave}. A full explanation of the poor convergence property and error analysis was later given by Diethelm \cite{diethelm2008investigation}. We also note a more recent exploration of positive and negative examples of these non-classical approaches in \cite{liu2019theoretical}. In \cite{diethelm2008investigation}, Diethelm argued for a modification of the Yuan--Agrawal method which substantially improved its convergence properties by using a transformation of the function $\phi_f : (0,\infty) \times [0,T]\rightarrow \mathbb{R}$ to a function $\bar\phi_f : (-1,1) \times [0,T]\rightarrow \mathbb{R}$ inspired by Gautschi's classical work on quadrature methods for half-infinite integrals \cite{gautschi1991quadrature}:
\begin{align}\label{eq:Diethelmphi}
\bar \phi_f(\kappa,t) := 2 (1-\kappa)^{-\bar\alpha} (1+\kappa)^{\bar\alpha-2} \phi_f\left(\frac{1-\kappa}{1+\kappa} , t \right).
\end{align}
We note the connection of this idea to Cayley transforms which map the right half plane to the unit disk, cf. \cite{bultheel2005generalizations}. In one dimension this changes the integral on the half-infinite Laguerre domain to an integral on the Jacobi domain:
\begin{align*}
\int_0^\infty \phi_f(w,t) \mathrm{d}w = \int_{-1}^{1} (1-\kappa)^{\bar\alpha} (1+\kappa)^{-\bar\alpha} \bar\phi_f(\kappa,t) \mathrm{d}\kappa,
\end{align*}
with $\bar{\alpha} := 2\alpha-2\lceil \alpha\rceil +1 \in (-1,1)$. Diethelm's method thus uses Gauss--Jacobi quadrature instead of Gauss--Laguerre quadrature which as analyzed in \cite{diethelm2008investigation} substantially improves the method's convergence in terms of the number of quadrature points needed.\\
In 2010 Birk and Song \cite{birk2010improved} published a paper which provides an alternative point of view on the Yuan--Agrawal and Diethelm methods. Rather than working in the time domain, Birk and Song derive a frequency domain non-classical method for Caputo derivatives. It is widely known, see e.g. \cite{kilbas1993fractional,tseng2000computation,birk2010improved}, that the Caputo derivative has the Fourier multiplier form
\begin{align}\label{eq:fouriermultiplierform}
\mathcal{F}\left\{\frac{\partial^\alpha}{\partial t ^\alpha} f(t)\right\}(\omega) = (i \omega)^\alpha \mathcal{F}\left\{f(t)\right\}(\omega),
\end{align}
if $\forall t < 0: f(t) = 0$, with the Fourier transform defined by
\begin{align*}
    \mathcal{F}\{ f(t) \}(\omega) = \frac{1}{\sqrt{2 \pi}} \int_\mathbb{R} e^{-i\omega t} f(t) \mathrm{d}t.
\end{align*}
Birk and Song \cite{birk2010improved} show that we can express $(i \omega)^\alpha$ by the following integral:
\begin{align}\label{eq:preapproxbirksong}
(i \omega)^\alpha = \frac{(-1)^{\lfloor \alpha \rfloor}2 \sin(\pi \alpha )}{\pi} \int_0^\infty \frac{(i\omega)^{\lceil \alpha \rceil}}{i\omega + p^2} p^{\tilde\alpha} \mathrm{d}p,
\end{align}
where $\tilde\alpha := 2\alpha - 2\lceil \alpha \rceil + 1$. Note that $\tilde\alpha \in (-1,1)$ for any non-integer order of fractional differentiation $\alpha >0$, $\alpha \notin \mathbb{N}$. This result has a natural interpretation as the frequency domain equivalent of the time domain Yuan--Agrawal form, cf. \cite[Appendix A]{birk2010improved}. Birk and Song \cite{birk2010improved} note that the approaches of Yuan--Agrawal and Diethelm can both be understood as resulting from a quadrature rule being applied to Equation \eqref{eq:preapproxbirksong}:
\begin{align}\label{eq:generalquadraturefrequencydomain}
(i\omega)^\alpha \approx (i \omega)^{\lceil \alpha \rceil} \sum_{j=1}^L \frac{A_j}{i\omega+s_j^2},
\end{align}
with $A_j$, $s_j$ taking different values depending on whether an $L$-point Gauss--Laguerre or Gauss--Jacobi quadrature is used, see Table 1. The third row in Table 1 corresponds to an alternative Gauss--Jacobi quadrature option discussed by Birk and Song in \cite{birk2010improved} which they show has advantages in some frequency domains and disadvantages in others.
\begin{table}[H]
\begin{center}
\begin{tabular}{c c c c}
\textbf{method} &    \textbf{quadrature type}      &      $A_j$       & $s_j$  \\ \hline \hline
Yuan--Agrawal \cite{yuan2002numerical}   & Gauss--Laguerre         & $\frac{(-1)^{\lfloor \alpha \rfloor} 2 \sin(\pi \alpha)}{\pi} e^{p_j} \lambda_j p_j^{\tilde\alpha}$ &$p_j$                     \\   
Diethelm \cite{diethelm2008investigation}   & $(1-x)^{\tilde\alpha}(1+x)^{-\tilde\alpha}$        &         $\frac{(-1)^{\lfloor \alpha \rfloor} \sin(\pi \alpha)}{\pi} \frac{4 \lambda_j}{(1+p_j)^2}$  & $\frac{1-p_j}{1+p_j}$    \\  
Birk--Song \cite{birk2010improved}    & $(1-x)^{2\tilde\alpha+1}(1+x)^{1-2\tilde\alpha}$       & $\frac{(-1)^{\lfloor \alpha \rfloor} \sin(\pi \alpha)}{\pi} \frac{8 \lambda_j}{(1+p_j)^4}$         & $\frac{(1-p_j)^2}{(1+p_j)^2}$   \\    
\end{tabular}
\caption{Parameter values for the $L$-point Gauss--Laguerre and respectively Gauss--Jacobi methods corresponding to Yuan--Agrawal's, Diethelm's and Birk--Song's methods in Equation \eqref{eq:generalquadraturefrequencydomain}. Note that $\lambda_j$ and $p_j$ in each row denote the appropriate weights and abscissae of the quadrature methods, cf. \cite[Section 4.2]{birk2010improved}.}
\end{center}
\label{tab:quadratureformulae}
\end{table}
\section{Towards a recursive non-classical sparse spectral method}\label{sec:recursivesection}
While the approach discussed in the previous section eliminates the requirement to memorize all previous states of the function $f(t)$ to which we apply the Caputo derivative, this is achieved by introducing additional internal variables. A recursive way to alleviate this drawback was suggested by Birk and Song \cite{birk2010improved}. They use Equation \eqref{eq:fouriermultiplierform} and the inverse Fourier transform of the rational approximation in \eqref{eq:generalquadraturefrequencydomain} to obtain
\begin{align}\label{eq:finalbirksong}
\frac{\partial^\alpha}{\partial t ^\alpha} f(t) &\approx \sum_{j=1}^L A_j \int_0^t e^{-s_j^2 (t-\tau)} \frac{\partial^{\lceil \alpha \rceil}}{\partial \tau^{\lceil \alpha \rceil}} f(\tau) \mathrm{d}\tau = \sum_{j=1}^L A_j \psi_j(t),
\end{align}
with
\begin{align}\label{eq:auxiliarydef}
\psi_j(t) := \int_0^t &e^{-s_j^2 (t-\tau)} \frac{\partial^{\lceil \alpha \rceil}}{\partial \tau^{\lceil \alpha \rceil}} f(\tau) \mathrm{d}\tau.
\end{align}
We note that one can arrive at an equivalent expression starting from Yuan--Agrawal's approach without needing to pass through frequency space. The integrals under the sum, i.e. $\psi_j(t)$, are readily found to satisfy the following recurrence relationship when advancing by a time step $\Delta t$:
\begin{align}
\psi_j(t) = e^{-s_j^2 \Delta t} \psi_j(t-\Delta t) + \int_{t-\Delta t}^t e^{-s_j^2 (t-\tau)} \frac{\partial^{\lceil \alpha \rceil}}{\partial \tau^{\lceil \alpha \rceil}} f(\tau) \mathrm{d}\tau.\label{eq:abstractrecurrencebirksong}
\end{align}
Note that we have $\psi_j(0) = 0$ and the right-hand side of \eqref{eq:abstractrecurrencebirksong} only depends on the integral which was already computed in the previous time step as well as an incrementing term without long term memory of either $f(t)$ or $\psi_j(t)$.\\
Birk and Song show in \cite{birk2010improved} how the results discussed until this point can be used to solve ODEs involving the Caputo derivative: For simplicity we use $\alpha \in (0,1)$ for the remainder of the paper such that $$\frac{\partial^{\lceil \alpha \rceil}}{\partial \tau^{\lceil \alpha \rceil}} = \frac{\partial}{\partial \tau}, \quad \alpha \in (0,1),$$ as the general $\alpha$ method becomes somewhat obfuscated by notation. Generalizations for higher orders of $\alpha$ are straightforward to obtain. \\
Approximating $f(t)$ via linear interpolation during each time step,  i.e.:
\begin{align}\label{eq:pseudotrapezoid}
f(\tau) \approx f((n-1) \Delta t) \left( 1 - \frac{\tau}{\Delta t} \right) + f(n \Delta t) \frac{\tau}{\Delta t}, \quad 0 \leq \tau \leq \Delta t,
\end{align}
one readily obtains the following approximation from \eqref{eq:abstractrecurrencebirksong}:
\begin{align}
\psi_j(n \Delta t) \approx \psi^n_j = e^{-s_j^2 \Delta t} \psi_j((n-1) \Delta t) + \tfrac{(1-e^{-s_j^2 \Delta t})}{s_j^2}\tfrac{f(n\Delta t)-f((n-1)\Delta t)}{\Delta t}.\label{eq:trapezoidal}
\end{align}
By the above discussion we thus have
\begin{align}\label{eq:caputoODE}
\left(\frac{\partial^\alpha f}{\partial t ^\alpha}\right) (n \Delta t) &\approx \sum_{j=1}^L A_j \left(e^{-s_j^2 \Delta t} \psi_j((n-1) \Delta t) + \tfrac{(1-e^{-s_j^2 \Delta t})}{s_j^2}\tfrac{f(n\Delta t)-f((n-1)\Delta t)}{\Delta t}\right).
\end{align}
\begin{remark}
Birk and Song \cite{birk2010improved} refer to the linear interpolation in \eqref{eq:pseudotrapezoid} as `trapezoidal rule' but while linear interpolation is behind the trapezoidal rule applied to $$\int_{t-\Delta t}^t f(\tau) \mathrm{d}\tau,$$ this is not equivalent to applying the trapezoidal rule to integrate $$\int_{t-\Delta t}^t e^{-s_j^2 (t-\tau)} \frac{\partial}{\partial \tau} f(\tau) \mathrm{d}\tau.$$\end{remark}
Naturally the linear interpolation can be replaced with a higher order interpolation to obtain a higher order method at the cost of storing additional function values. In fact any numerical method capable of numerically integrating the right-most term of \eqref{eq:abstractrecurrencebirksong} can be substituted in this approach. In this paper we restrict ourselves to using linear interpolation for notational simplicity but recommend using higher order methods in step size sensitive applications.\\
Since our aim is to derive a memory free solver for PDEs featuring time-fractional Caputo derivatives the minimal problem of interest involves $\frac{\partial^\alpha}{\partial t ^\alpha} f(t,\mathbf{x})$, where $f$ now depends on both space $\mathbf{x} \in \Omega \subset \mathbb{R}^d$ and time $t \in [0,T]$. To obtain a method in which each time step involves the solution of a \emph{sparse} linear system, while simultaneously providing excellent spatial convergence properties we choose to combine Birk--Song's recursive method with sparse spectral methods as described in Section \ref{sec:spectralintro}.\\
At any fixed point in time $n \Delta t$ we represent the function $f(t,\mathbf{x})$ and the $L$ auxiliary node functions $\psi_j(t,\mathbf{x})$ via their coefficient vectors in an orthogonal polynomial basis $\mathbf{P}(\mathbf{x})$ on the desired domain, i.e.:
$$f(n \Delta t,\mathbf{x}) = \mathbf{P}(\mathbf{x}) \bm{f}(n\Delta t),$$
$$\psi_j(n \Delta t,\mathbf{x}) = \mathbf{P}(\mathbf{x}) \bm{\psi}_j(n\Delta t).$$
We store only the current states of these functions such that the coefficient vectors $\bm{f}^n$ and $\bm{\psi}^n_j$ approximating $\bm{f}(n\Delta t)$ and $\bm{\psi}_j(n\Delta t)$ are overwritten in each step. By Equation \eqref{eq:caputoODE} we can thus approximate the Caputo derivative (with $\alpha \in (0,1)$) at a point in time $n \Delta t$ as
\begin{align}\label{eq:pdeversionofapprox}
\left(\frac{\partial^\alpha f}{\partial t ^\alpha}\right) (n \Delta t, \mathbf{x}) &\approx \mathbf{P}(\mathbf{x}) \sum_{j=1}^L A_j \left(e^{-s_j^2 \Delta t} \bm{\psi}^{n-1}_j + \tfrac{(1-e^{-s_j^2 \Delta t})}{s_j^2 \Delta t}\left(\bm{f}^n-\bm{f}^{n-1}\right)\right),
\end{align}
with the following recurrence for the auxiliary function coefficients:
\begin{align}\label{eq:caputoPDEpsi}
\psi(n \Delta t,\mathbf{x}) = \mathbf{P}(\mathbf{x})\bm{\psi}(n\Delta t) \approx \mathbf{P}(\mathbf{x}) \bm{\psi}_j^n = \mathbf{P}(\mathbf{x})\left(e^{-s_j^2 \Delta t} \bm{\psi}_j^{n-1} + \tfrac{(1-e^{-s_j^2 \Delta t})}{s_j^2}\tfrac{\bm{f}^n-\bm{f}^{n-1}}{\Delta t}\right).
\end{align}
To illustrate the method in action we finish this section by working through a toy problem for time-fractional PDEs with constructed analytic solutions. The following time-fractional equation of motion is an ODE arising in the massless one-dimensional Kelvin-Voigt model:
\begin{align}\label{eq:odenumericsexample}
kf + c \frac{\partial^\alpha}{\partial t^\alpha} f= g,\\
f(0) = 0,
\end{align}
where $c > 0$, $k > 0$ and the external force $g(t)$ is given by
\begin{align*}
g(t) = \begin{cases} 0, \quad \text{if } t=0,\\ 1, \quad \text{if } t > 0. \end{cases}
\end{align*}
This equation has the following known analytic solution, cf. \cite{schmidt2006critique,diethelm2008investigation,birk2010improved}
\begin{align*}
f(t) = \frac{1}{k} \left(1 - E_{\alpha,1}\left( - \frac{k t^\alpha}{c} \right) \right) .
\end{align*}
Equation \eqref{eq:odenumericsexample} was part of a critique of non-classical methods due to Schmidt and Gaul \cite{schmidt2006critique} and was also addressed again by Diethelm in \cite{diethelm2008investigation} as well as by Birk and Song in \cite{birk2010improved}. With its known explicit solution and relevant historical role in early papers on nonclassical methods this equation has become a popular toy model for numerically evaluating the accuracy of fractional Caputo differential equation solvers. We will use it to derive fractional PDEs with known solutions to use as toy models to test the PDE approach with before moving on to more involved problems in the numerical experiments in Section \ref{sec:numexpsec}. Multiplying Equation \eqref{eq:odenumericsexample} by $e^{-x}$, we readily obtain
\begin{align}
k\frac{\partial}{\partial x} f + c \frac{\partial^\alpha}{\partial t^\alpha} f &= g,\label{eq:pdenumericsexamples1}
\end{align}
satisfying $f(0,x) = 0$ and where we have defined 
\begin{align*}
g(t,x) := g(t)e^{x} = \begin{cases} 0, \hspace{5mm} \text{if } t=0,\\ e^{x}, \quad \text{if } t > 0. \end{cases}
\end{align*}
An explicit solution on $\mathbb{R}$ can be inferred via the equation's relationship with \eqref{eq:odenumericsexample}:
\begin{align*}
f(t,x) &= f(t)e^{-x} = \frac{e^{x}}{k} \left(1 - E_{\alpha,1}\left( - \frac{k t^\alpha}{c} \right) \right).
\end{align*}
Using Equation \eqref{eq:pdeversionofapprox} one can then derive the following history-free nonclassical scheme for Equation \eqref{eq:pdenumericsexamples1}:
\begin{align*}
\mathbf{P}(x) &{} \left( k \mathcal{D} +  \left(c \sum_{j=1}^L A_j \frac{(1-e^{-s_j^2 \Delta t})}{s_j^2 \Delta t} \right)\mathcal{I}  \right) \bm{f}^n \\&=  \mathbf{P}(x) \left( \bm{g}^n + \left(c \sum_{j=1}^L A_j \frac{(1-e^{-s_j^2 \Delta t})}{s_j^2 \Delta t} \right)\bm{f}^{n-1} - \left( c \sum_{j=1}^L A_j e^{-s_j^2\Delta t} \bm{\psi}_j^{n-1}\right) \right),
\end{align*}
where the $\bm{\psi}_j^{n-1}$ are computed recursively via \eqref{eq:caputoPDEpsi}. More precisely, to obtain a \emph{banded} linear system we use the following bases and operators:
\begin{align*}
&\mathbf{P}^{(1,1)}(x) \left( k \mathcal{D} +  \left(c \sum_{j=1}^L A_j \frac{(1-e^{-s_j^2 \Delta t})}{s_j^2 \Delta t} \right)\mathcal{C}  \right) \bm{f}^n \\&=  \mathbf{P}^{(1,1)}(x) \mathcal{C}  \left( \bm{g}^n + \left(c \sum_{j=1}^L A_j \frac{(1-e^{-s_j^2 \Delta t})}{s_j^2 \Delta t} \right)\bm{f}^{n-1} - \left( c \sum_{j=1}^L A_j e^{-s_j^2\Delta t} \bm{\psi}_j^{n-1}\right) \right),
\end{align*}
where $\mathcal{C}$ is the banded conversion operator (a change of basis) from the Legendre polynomials $\mathbf{P}^{(0,0)}(x)$ to the Jacobi polynomials $\mathbf{P}^{(1,1)}(x)$ and $\bm{f}$ and $\bm{\psi}$ are expanded in $\mathbf{P}^{(0,0)}(x)$. This change of basis operator is required since for $\mathcal{D}$ to be banded in the ultraspherical spectral method it has to map from the basis $\mathbf{P}^{(0,0)}(x)$ to $\mathbf{P}^{(1,1)}(x)$ as discussed in Section \ref{sec:spectralintro}, cf. \cite{olver2020fast,olver2013fast} and \cite[18.9]{nist_2022}. As usual for this type of spectral method, the boundary conditions at each time step can be enforced either by using suitable weighted bases or by appending rows corresponding to the evaluation operator to the top of the LHS as well as the boundary condition values to the top of the RHS, see e.g. the methods discussed in \cite{townsend2015automatic,hale2018fast,gutleb2021fast} and the references therein. We opt to use the latter option along with the analytic solution values as both the $t=0$ initial and spatial boundary conditions in this numerical experiment.\\
In Figure \ref{fig:toycontour} we show absolute and relative error contour plots for the problem in \eqref{eq:pdenumericsexamples1} with $k=10$ and $c=100$ for a degree $K=40$ approximation computed with $L=50$ quadrature points and $\Delta t = 2^{-20}$ for $t \in [0,1]$, $x \in [-1,1]$. Figure \ref{fig:toyslice} shows the error for $x \in [-1,1]$ at time $t=1$. Stepping through $1000$ time steps with $L = 50$ and $K = 40$ takes roughly 29 milliseconds on a Lenovo ThinkPad X1 Carbon laptop with a 12th Gen Intel i5-1250P CPU.
\begin{remark}
We note that while in the worked example this section we let $c$ and $k$ be constants independent of $t$ and $x$ due to the availability of analytic solutions, the nonclassical method we propose in this paper readily accommodates spacetime dependent coefficients by e.g. replacing the scalar coefficient $k(t,x)$ by a multiplication operator $\mathrm{K}^n(\mathrm{X})$ at time $n \Delta t$. If $k(t,x)$ is polynomial in $x$, this operator does not break the sparsity of the method since polynomial multiplication operators are banded, cf. \cite{olver2020fast}. Well-behaved non-polynomial coefficient functions, while in theory resulting in dense multiplication operators, may in certain circumstances nevertheless decay rapidly off-band and can thus sometimes still be approximated by banded matrices. If $k(t,x)=k(x)$ is independent of $t$ the banded multiplication operator may be stored in memory for significant performance improvements.
\end{remark}
\begin{figure}\centering
     \subfloat[absolute error]
    {{ \centering \includegraphics[width=9.5cm]{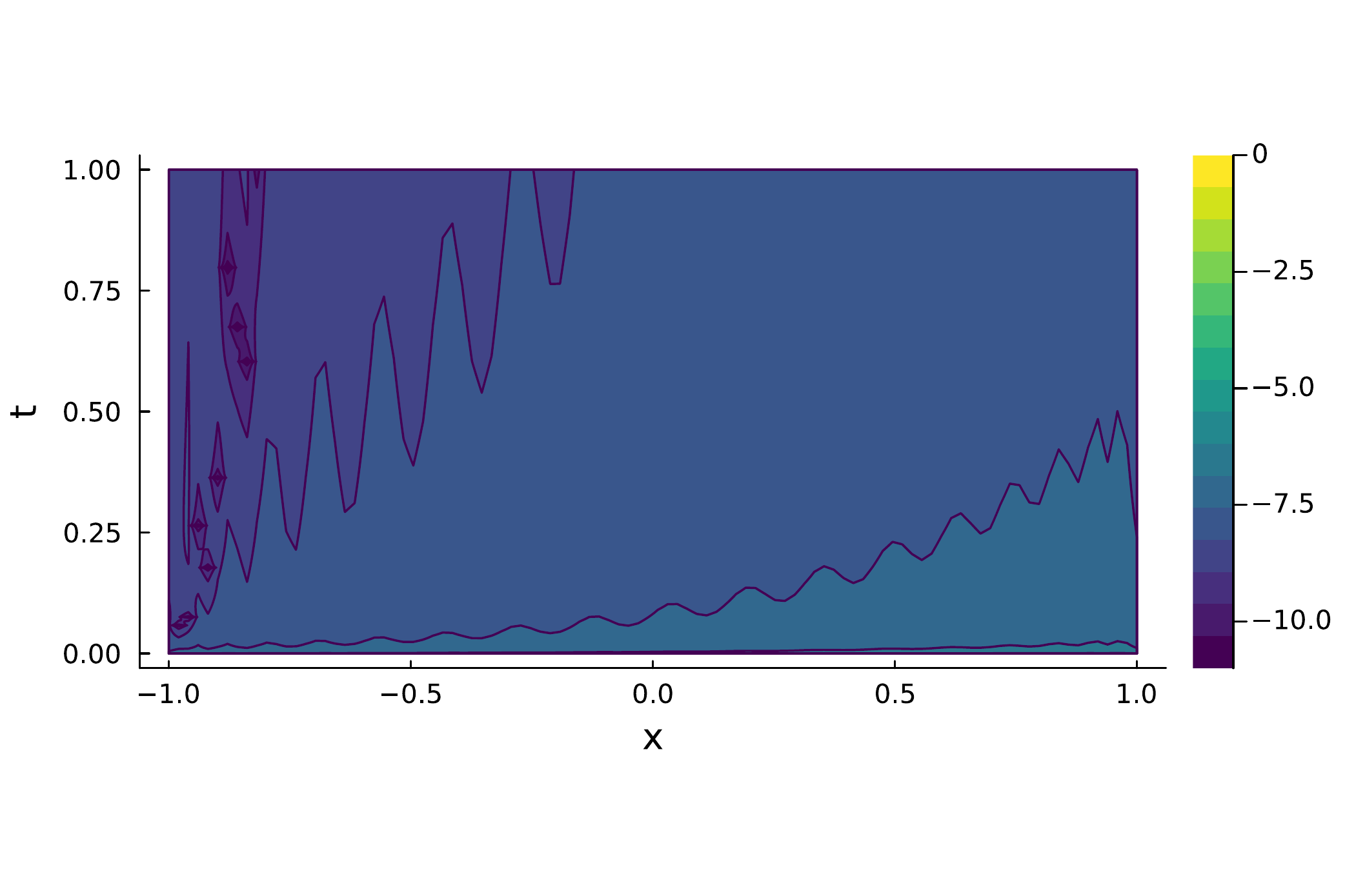} }}\\
     \subfloat[relative error]
    {{ \centering \includegraphics[width=9.5cm]{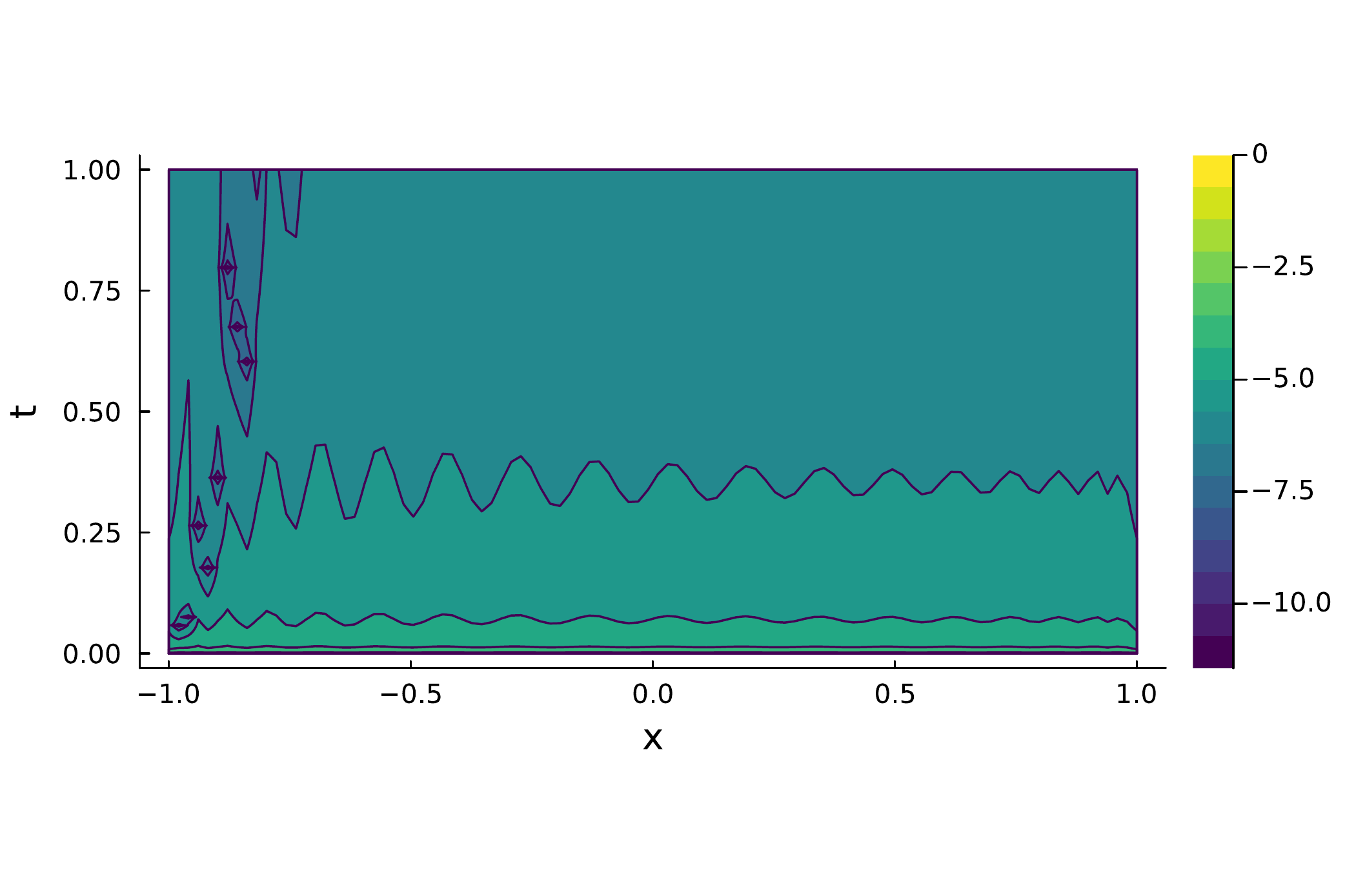} }}
    \caption{Contour plots of the absolute and relative errors for the fractional PDE in Equation \eqref{eq:pdenumericsexamples1} with $k=10$, $c=100$ on the spacetime box $[0,1]\times[-1,1]$. The legend is logarithmic, indicating the order of magnitude of the errors. The numerical solution was computed with $L=50$ quadrature points, approximation degree $K=40$ and with stepsize $\Delta t = 2^{-20}$. See Figure \ref{fig:toyslice} for a slice through $t=1$.}
    \label{fig:toycontour}
    \end{figure}
        \begin{figure} \centering \includegraphics[width=8.5cm]{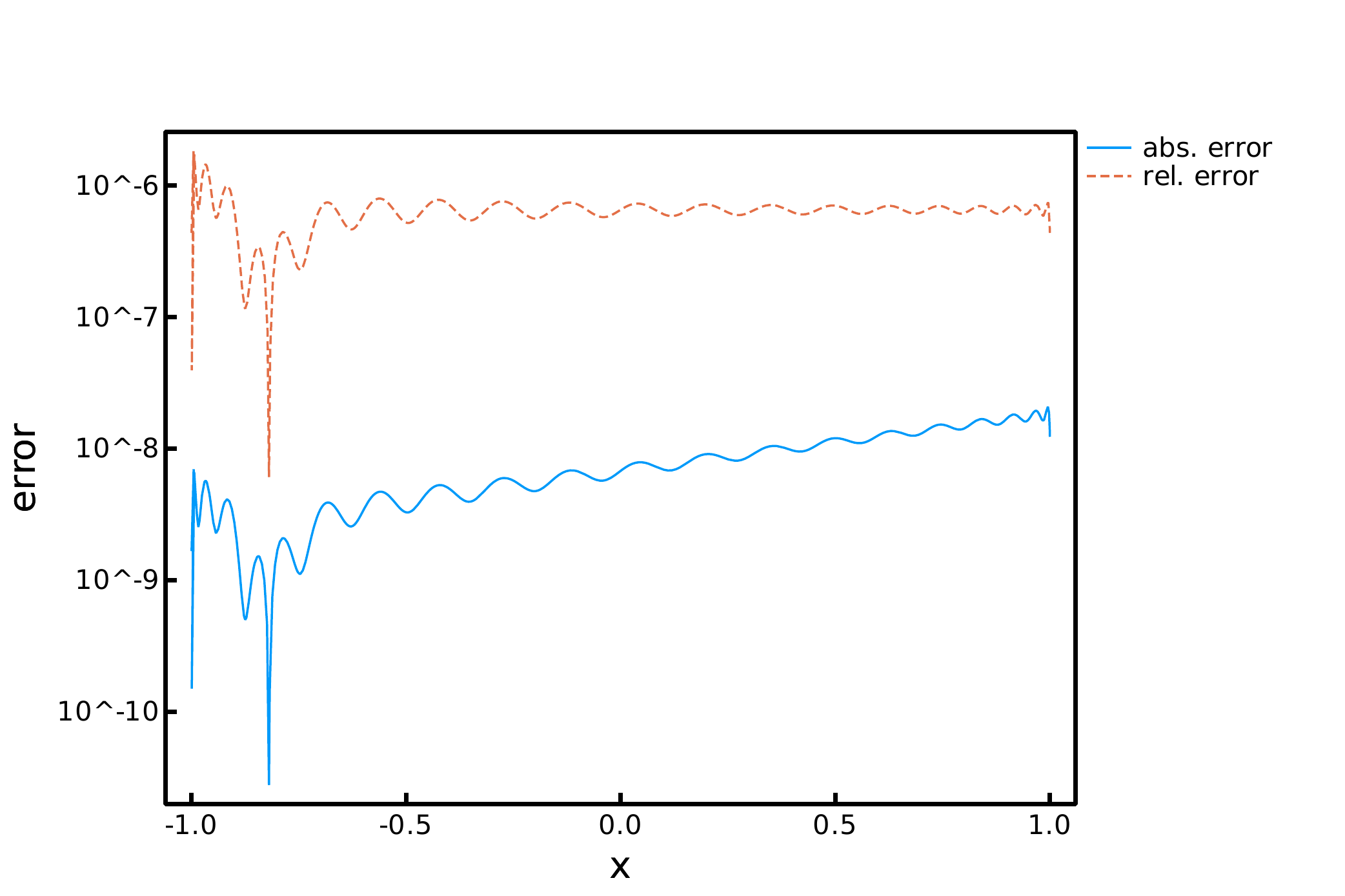} 
    \caption{Absolute and relative error slice at $t=1$ for the experiment in Figure \ref{fig:toycontour}.}
    \label{fig:toyslice}
    \end{figure}
As a final remark for this section, we note that the domain $\Omega \subset \mathbb{R}^d$ determines what constitutes a suitable choice of basis functions and that the method can thus be used in any dimensions as long as a basis suitable for computation is available on the desired domain (although for dimensions $\geq 5$ the curse of dimensionality may cause issues without use of further structure such as symmetries). Canonical higher dimensional domains such as higher dimensional cuboids can be treated using tensor products of Jacobi polynomials while arbitrary dimensional ball domains can make use of higher dimensional generalizations of Zernike polynomials. For further examples we refer to the literature overview in Section \ref{sec:spectralintro}.\\
As we have seen, to ensure that each step is a sparse linear solve, the left and right-hand side orthogonal polynomial bases must be chosen carefully such that any other linear operators appearing in the fractional partial differential equation are accounted for consistently, typically also requiring sparse conversion operators in the process. Due to the many possible combinations of linear operators one may encounter when dealing with fractional PDEs, we show the procedure of ensuring sparsity in more cases in the numerical experiments in Section \ref{sec:numexpsec}.
\section{Memory cost of the history-free method}\label{sec:memorycost}
In this section we investigate how much memory the introduced static memory method requires. We focus only on the memory cost strictly associated with computing the Caputo fractional derivative as the exact cost of the method will otherwise naturally vary depending on the other present differential or integral operators. We count the minimal required values in storage for a solution loop in a diagram in Figure \ref{fig:memoryrequirement}, where $K$ is the length of the coefficient vectors in the orthogonal polynomial basis of choice and $L$ is the number of quadrature points chosen in Eq. \eqref{eq:pdeversionofapprox}. The number of values required in storage is found to be $L(2 + K)+2 K$ and we observe that this memory requirement is entirely independent of the number of time steps taken to reach the final desired time $T=N \Delta t$. We also note that for most problems of interest one expects $L \ll N$. In practice the method's efficiency and accuracy can benefit from keeping some additional repeatedly used static values in memory but this additional cost is also $O(L)$ rather than $O(N)$, which is the typical cost associated with classical methods cf. \cite{jin2023numerical}. Whether the obtained memory footprint is smaller than that of a short memory principle approach will vary on a problem by problem basis, as the cutoff at which the incurred error remains sensible varies -- the question boils down to how the smallest sensible memory cutoff $N_\ell$ compares to the required number of quadrature points $L$ (note that nonclassical methods do not rely on short term memory cutoffs) and how much is known about the expected behavior of the solutions.
\begin{figure}[htbp]
\begin{center}
\includegraphics[width=9cm]{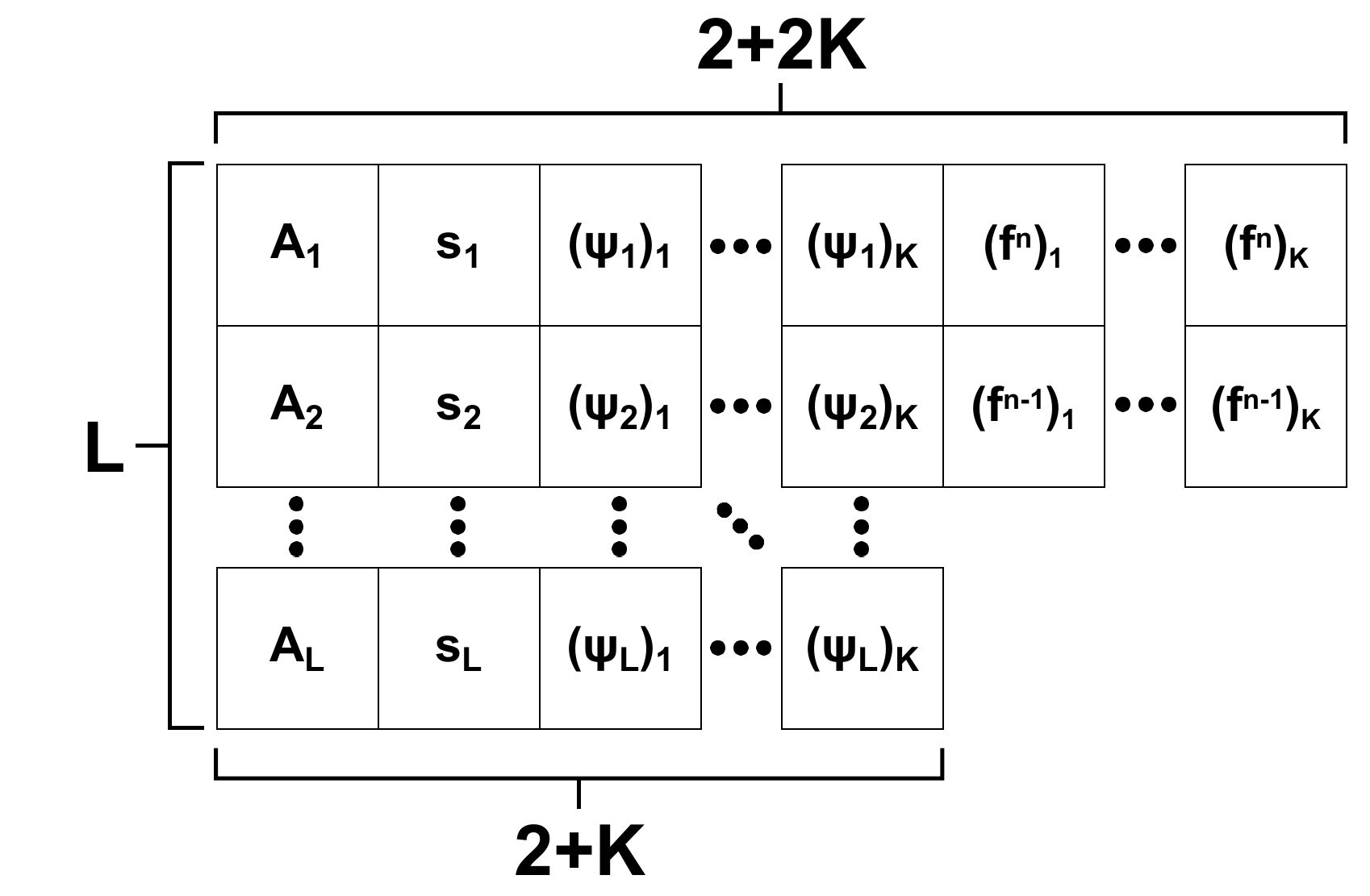}
\caption{Diagram of memory cost associated with solving the Caputo derivative part of a fractional PDE using the method described in this paper, where $K$ is the length of the coefficient vectors in the orthogonal polynomial basis of choice and $L$ is the number of quadrature points chosen in Eq. \eqref{eq:pdeversionofapprox}. The number of values required in storage is $L(2 + K)+2 K$.}
\label{fig:memoryrequirement}
\end{center}
\end{figure}
\section{Error analysis and computational complexity}\label{sec:error}
In this section we provide a discussion of the error as well as computational complexity associated with our method. A general discussion of the error of both the original Yuan-Agrawal method and Diethelm's modification was already given by Diethelm himself in Theorem 11 of \cite{diethelm2008investigation}.
Note that, disregarding the independent spatial error, Diethelm's proof assumes that the differential equation in Eq. \eqref{eq:diffeqphiYA} is solved for the auxiliary functions rather than the more efficient recurrence relationship approach. If we were to assume that the recursive approach performs at least as well as a differential equation solver in each step then Diethelm's error estimate holds similarly for our suggested solver as long as the independent spatial error is also taken into account. The numerical stability of recurrence relationships is a subtle matter which has received a large amount of attention in the literature, see e.g. \cite{olver1972note}. In Section \ref{sec:numexpsec} we will present some numerical experiments regarding the stability of the recurrence relationship to convince ourselves of the applicability of the proposed method. \\
In what follows we provide an error analysis for the recursive approach to time-fractional Caputo derivatives including resolving spatial dependence in a basis of orthogonal polynomials. The results which follow apply for the direct computation of the Caputo derivative rather than fractional differential equations, keeping in mind that for a particular fractional PDE the error will also be affected by what method we use to deal with the remaining differential or integral operators. For the sake of brevity we also only discuss Diethelm's quadrature rule and fractional orders $0 \leq \alpha \leq 1$ but note that analogous results for the Birk--Song quadrature rule and higher orders can be derived in the same way.
\begin{lemma}\label{lem:quaderror}
Let $\bar{P}_n^{(\bar{\alpha},-\bar{\alpha})}(x)$ with $\bar{\alpha} := 2\alpha-2\lceil \alpha\rceil +1 \in (-1,1)$ denote the monic Jacobi polynomial of order $n$ and $\langle\cdot,\cdot\rangle_{(a,b)}$ denote the inner product with respect to the corresponding Jacobi weight, i.e.:
$$\langle f,g \rangle_{(a,b)} := \int_{-1}^1 f(t) g(t) (1-t)^{a}(1+t)^{b} \mathrm{d}t.$$
Furthermore, let $\bar{\phi}_{T,\mathbf{x}}(\kappa) := \bar{\phi}_f(\kappa,T,\mathbf{x})$ with $\bar{\phi}_f$ defined in Eq. \eqref{eq:Diethelmphi}. If $\bar{\phi}_{T,\mathbf{x}} \in C^{2L}[0,T]$ then the quadrature error incurred from Diethelm's method as described in Section \ref{sec:previouswork} is given by
\begin{align*}
\left(\frac{\partial^\alpha}{\partial t ^\alpha} f\right)(T,\mathbf{x}) - \sum_{j=1}^L A_j \psi_j(T,\mathbf{x}) =  \frac{\bar{\phi}_{T,\mathbf{x}}^{(2L)}(\xi)}{(2L)!} \langle \bar{P}^{(\bar{\alpha},-\bar{\alpha})}_L,\bar{P}^{(\bar{\alpha},-\bar{\alpha})}_L\rangle_{(\bar{\alpha},-\bar{\alpha})},
\end{align*}
for some $\xi \in [0,T]$.
\end{lemma}
\begin{proof}
It is a classical result, cf. \cite[Theorem 3.6.24]{stoer1996introduction}, that if $g \in C^{2L}[a,b]$ quadrature methods with monic polynomials $\bar{P}_n(x)$ orthogonal with respect to $w(x)$ on $[a,b]$ satisfy
\begin{align*}
\int_a^b w(x) g(x) \mathrm{d}x - \sum_{j=1}^{L} \lambda_j g(p_j) = \frac{g^{(2L)}(\xi)}{(2L)!} \langle \bar{P}_L,\bar{P}_L \rangle,
\end{align*}
for some $\xi \in (a,b)$ where $\langle \cdot, \cdot \rangle $ is the inner product with respect to the appropriate orthogonality weight. To use this result we have to translate our compact notation into a sligtly more cumbersome one: Using Table 1 and the definitions in Section \ref{sec:previouswork} we find that for Diethelm's method
\begin{align*}
\sum_{j=1}^L A_j \psi_j(T,\mathbf{x})  =  \sum_{j=1}^L \lambda_j\left(\frac{(-1)^{\lfloor \alpha \rfloor} \sin(\pi \alpha)}{\pi} \frac{4\psi_j(T,\mathbf{x})}{(1+p_j)^2}\right)
&= \sum_{j=1}^L \lambda_j \bar{\phi}_f(p_j,T,\mathbf{x}).
\end{align*}
The stated error expression follows as an immediate corollary.
\end{proof}
\begin{definition}
We denote by $\mathcal{P}_K$ the canonical projection operator corresponding to truncation-to-degree-$K$, i.e.: $$\mathbf{P}(\mathbf{x})\mathcal{P}_K \bm{g} = \sum_{k=0}^K g_k P_k(\mathbf{x}).$$
\end{definition}
\begin{definition}
We denote by $\mathrm{err}_{\mathbf{P},K}(g(\mathbf{x}))$ the error incurred from truncating the polynomial expansion of a function $g$ in the basis $\mathbf{P}(\mathbf{x})$ at degree $K$, i.e.:
\begin{align*}
\mathrm{err}_{\mathbf{P},K}(g(\mathbf{x})) = | \mathbf{P}(\mathbf{x})\bm{g}   - \mathbf{P}(\mathbf{x})\mathcal{P}_K \bm{g}   |.
\end{align*}
\end{definition}
While for sufficiently nice functions one generally expects spectral convergence in the truncation degree when expanding functions in orthogonal polynomials, cf. \cite[Chapter 21]{trefethen2019approximation} and \cite[Section 2.3]{boyd2001chebyshev}, more generally the study of these truncation errors is an active field of research in its own right, see \cite{hamzehnejad2022improved,wang2018new,wang2021much,zhao2013sharp,xiang2012error,wang2012convergence,wang2016optimal} and the references therein. We will make no attempt to fully cover the range of different error bounds for all the various different basis choices and domains our method can be applied with and instead settle for an abstract result in which the error caused by the polynomial approximation in space is separated from the time-fractional Caputo contribution.
\begin{proposition}\label{prop:errorjacobiabstract}
Let $f_\mathbf{x}(t) := f(t,\mathbf{x})$ with $ f_\mathbf{x} \in C^{1 +\left\lceil \alpha \right\rceil}[0,T]$ and $\bar{\phi}_{T,\mathbf{x}}(\kappa) := \bar{\phi}_f(\kappa,T,\mathbf{x})$ with $\bar{\phi}_{T,\mathbf{x}} \in C^{2L}[0,T]$. Furthermore, let $\bm{f}_\alpha^{N}$ denote the coefficient vector such that $$\left(\frac{\partial^\alpha}{\partial t ^\alpha}f\right)(T,\mathbf{x}) = \mathbf{P}(\mathbf{x})\bm{f}_\alpha^{N}.$$ We assume the quadrature-related constants $A_j$ and $s_j$ are obtained with negligible error and define the following non-negative functions:
\begin{align*}
M(\mathbf{x}) &= \max_{\xi_1 \in [0, T]} \left| \bar{\phi}_T^{(2L)}(\xi_1,\mathbf{x})\langle P^{(\bar{\alpha},-\bar{\alpha})}_L,P^{(\bar{\alpha},-\bar{\alpha})}_L\rangle_w \right|,\\
C(\mathbf{x}) &=   \tfrac{1}{2}\max_{j}|A_j| \max_{\xi_2 \in [0,T]} \left| \left(\tfrac{\partial^2}{\partial t^2}f\right)(\xi_2,\mathbf{x})\right|.
\end{align*}
Then the error incurred from using the recursive degree-$K$-truncated spectral method as described in Section \ref{sec:recursivesection} to compute the time-fractional Caputo derivative of order $0 \leq \alpha \leq 1$ of a function $f(t,\mathbf{x})$ at time $T$ using Diethelm's quadrature rule and a uniform time grid of step size $\Delta t = \frac{T}{N}$ is bounded by
\begin{align*}
&\left| \left(\frac{\partial^\alpha}{\partial t ^\alpha}f\right)(T,\mathbf{x}) - \mathbf{P}(\mathbf{x})\mathcal{P}_K\sum_{j=1}^L A_j  \bm{\psi}_j^N   \right| \\ & \qquad \qquad \leq \frac{M(\mathbf{x})}{(2L)!} + C(\mathbf{x}) L T \Delta t + \mathrm{err}_{\mathbf{P},K}\left(\sum_{j=1}^L A_j \psi_j(T,\mathbf{x})\right).
\end{align*}
\end{proposition}
\begin{proof}
Using the triangle inequality we obtain
\begin{align*}
&\left| \left(\frac{\partial^\alpha}{\partial t ^\alpha}f\right)(T,\mathbf{x}) - \mathbf{P}(\mathbf{x})\mathcal{P}_K\sum_{j=1}^L A_j  \bm{\psi}_j^N   \right| = \left| \mathbf{P}(\mathbf{x})\bm{f}_\alpha^{N} - \mathbf{P}(\mathbf{x})\mathcal{P}_K\sum_{j=1}^L A_j  \bm{\psi}_j^N   \right| \\ & \qquad \leq \left| \mathbf{P}(\mathbf{x})\bm{f}_\alpha^{N} - \mathbf{P}(\mathbf{x})\sum_{j=1}^L A_j  \bm{\psi}_j^N     \right|  + \left| \mathbf{P}(\mathbf{x})\sum_{j=1}^L A_j  \bm{\psi}_j^N   - \mathbf{P}(\mathbf{x})\mathcal{P}_K\sum_{j=1}^L A_j \bm{\psi}_j^N   \right|. 
\end{align*}
Note that the first term, $| \mathbf{P}(\mathbf{x})\bm{f}_\alpha^{N} - \mathbf{P}(\mathbf{x})\sum_{j=1}^L A_j  \bm{\psi}_j^N     | $, now describes an error contribution at each spatial point $\mathbf{x}$ which is solely incurred due to the recursive time-fractional Caputo derivative method. The second term is simply the approximation error of truncating the polynomial expansion of the function $\sum_{j=1}^L A_j \psi_j(T,\mathbf{x})$ in $\mathbf{P}(\mathbf{x})$ which we leave abstract.\\
To obtain a more explicit bound for the first term
we use Lemma \ref{lem:quaderror} to observe that the error added to the series approximation in \eqref{eq:finalbirksong} caused by perturbations $\delta_j(\mathbf{x})$ in the auxiliary functions $\psi_j(t,\mathbf{x})$ is
\begin{align*}
\left(\frac{\partial^\alpha}{\partial t ^\alpha} f\right)(T,\mathbf{x}) &- \sum_{j=1}^L A_j (\psi_j(T,\mathbf{x}) + \delta_j(\mathbf{x})) \\ &= \left(\frac{\partial^\alpha}{\partial t ^\alpha} f\right)(T,\mathbf{x}) - \sum_{j=1}^L A_j \psi_j(T,\mathbf{x}) - \sum_{j=1}^L A_j \delta_j(\mathbf{x})\\
&= \frac{\bar{\phi}_{T,\mathbf{x}}^{(2L)}(\xi)}{(2L)!} \langle P^{(\bar{\alpha},-\bar{\alpha})}_L,P^{(\bar{\alpha},-\bar{\alpha})}_L\rangle_w - \sum_{j=1}^L A_j \delta_j(\mathbf{x}).
\end{align*}
for some $\xi \in [0,T]$. From this we obtain the error bound
\begin{align*}
&\left| \left(\frac{\partial^\alpha}{\partial t ^\alpha} f\right)(T,\mathbf{x}) - \sum_{j=1}^L A_j (\psi_j(T,\mathbf{x}) + \delta_j(\mathbf{x})) \right|\\ & \qquad \leq \max_{\xi_1 \in [0, T]}  \left| \frac{\bar{\phi}_{T,\mathbf{x}}^{(2L)}(\xi_1)}{(2L)!} \langle P^{(\bar{\alpha},-\bar{\alpha})}_L,P^{(\bar{\alpha},-\bar{\alpha})}_L\rangle_w \right| + L  \max_{j}\left( \left|A_j \right| \left|\delta_j(\mathbf{x}) \right| \right).
\end{align*}
We thus seek a bound on $|\delta_j(\mathbf{x})|$ caused by the approximation in \eqref{eq:trapezoidal}. Comparing \eqref{eq:abstractrecurrencebirksong} with the approximation in \eqref{eq:trapezoidal} we find that the $\delta_j$ are the errors caused by integrating the derivative of the linear interpolation. If $f(x)$ is in $C^{q+1}[a,b]$ and $p_{f,q}(x)$ is its interpolating polynomial of degree $\leq q$ at $q+1$ points $x_0, ..., x_q \in [a,b]$, then the remainder term of the interpolation is given by \cite[3.3.11]{nist_2022}
\begin{align*}
f(x) - p_{f,q}(x) = \tfrac{\prod_{i=0}^{q}(x-x_i)}{(q+1)!} f^{(q+1)}(\xi),
\end{align*}
for some $\xi \in [a,b]$. For our case of taking derivatives of the linear interpolation in each time step we thus obtain that there exists a $\xi \in [(n-1)\Delta t,n\Delta t]$ such that
\begin{align*}
    \left(\tfrac{\partial}{\partial \tau}  f \right)(\tau,\mathbf{x})  - 
 \tfrac{f(n\Delta t) - f((n-1)\Delta t)}{\Delta t}  = \tfrac{\Delta t +2 (\tau - n\Delta t) }{2}  \left(\tfrac{\partial^2}{\partial t^2}f\right)(\xi,\mathbf{x}).
 \end{align*}
 Since for all $\tau \in [(n-1)\Delta t,n\Delta t]$ we have $|\tfrac{\Delta t +2(\tau - n\Delta t)}{2}| \leq \tfrac{\Delta t}{2}$ this yields the error bound
 \begin{align*}
\left| \left(\tfrac{\partial}{\partial \tau}  f \right)(\tau,\mathbf{x})  - 
 \tfrac{f(n\Delta t) - f((n-1)\Delta t)}{\Delta t} \right| \leq  \tfrac{\Delta t}{2} \max_{\xi \in [(n-1)\Delta t,n\Delta t]}\left|\left(\tfrac{\partial^2}{\partial t^2}f\right)(\xi,\mathbf{x})\right|.
\end{align*}
In a single time step the error in the integral is thus bounded by
\begin{align*}
\left|\int_{(n-1)\Delta t}^{n\Delta t} \right. &{} \left. e^{-s_j^2 (n\Delta t-\tau)}  \left(\left(\tfrac{\partial}{\partial \tau}  f \right)(\tau,\mathbf{x})  - 
 \tfrac{f(n\Delta t) - f((n-1)\Delta t)}{\Delta t} \right) \mathrm{d}\tau\right|\\
 &\leq \frac{\Delta t}{2}\max_{\xi \in [(n-1)\Delta t,n\Delta t]}\left|\left(\tfrac{\partial^2}{\partial t^2}f\right)(\xi,\mathbf{x})\right|\int_{(n-1)\Delta t}^{n\Delta t} e^{-s_j^2 (n\Delta t-\tau)}\mathrm{d}\tau\\
 & = \frac{\Delta t}{2}  \frac{(1-e^{-s_j^2\Delta t })}{s_j^2} \max_{\xi \in [(n-1)\Delta t,n\Delta t]}\left|\left(\tfrac{\partial^2}{\partial t^2}f\right)(\xi,\mathbf{x})\right|\\
  &\leq \frac{(\Delta t)^2}{2}  \max_{\xi \in [(n-1)\Delta t,n\Delta t]}\left|\left(\tfrac{\partial^2}{\partial t^2}f\right)(\xi,\mathbf{x})\right|,
\end{align*}
resulting in the following error bound for $\delta_j(\mathbf{x})$:
\begin{align*}
    \left| \delta_j(\mathbf{x}) \right| \leq \frac{N (\Delta t)^2}{2} \max_{\xi \in [0,T]}\left|\left(\tfrac{\partial^2}{\partial t^2}f\right)(\xi,\mathbf{x})\right| = \frac{T \Delta t}{2} \max_{\xi \in [0,T]}\left|\left(\tfrac{\partial^2}{\partial t^2}f\right)(\xi,\mathbf{x})\right|.
\end{align*}
Combining this with the above-derived error bounds yields the stated result.
\end{proof}
\begin{remark}
We note that similarly to what was discussed in \cite[Section 3.7]{stoer1996introduction}, such error bounds for quadrature-type methods tend to not be useful for error estimations in actual applications as they include difficult to compute or estimate high order derivative terms. A more practical approach to estimating errors in an application is to check convergence with higher precision results. Nevertheless, the derived error bounds suggest a sensible convergence strategy: Prioritize convergence in $L$, while keeping $L\Delta t$ constant, then adjust $\Delta t$ further if required once convergence in $L$ has been achieved. We verify that this is a sensible convergence strategy in practice in the numerical experiments in the next section.
\end{remark}
As mentioned before, it is straightforward to see from the proof of Proposition \ref{prop:errorjacobiabstract} that choosing higher order accuracy integration methods for the integrals in the recurrence instead of using \eqref{eq:pseudotrapezoid} would also lead to a higher order in $\Delta t$ method for the Caputo derivative.\\

Finally, we address computational complexity and the associated memory footprint. We will suppose here that the desired accuracy is reached with $K$ polynomial coefficients and $L$ quadrature points. First we note that as seen in the numerical experiments in Section \ref{sec:numexpsec} and the worked example in Section \ref{sec:recursivesection}, each time step involves two computations: First, the update of the auxiliary function via the recurrence and second the linear system solve. Additionally, at the beginning of each computation the quadrature-related constants must be computed, although this only has to be done a single time and will thus be considered pre-computation. The recursive update of the auxiliary function is $O(1)$ for each quadrature point and spatial index, cf. Figure \ref{fig:memoryrequirement}, meaning that the computational cost of this update has computational cost $O(KL)$. Meanwhile, the linear solve involves a sparse (banded or block-banded) operator of size $K\times K$ if set up correctly using the ultraspherical spectral method, cf. \cite{olver2013fast,olver2020fast}. The exact computational cost of the second part depends on this sparsity pattern but if it is banded then one expects $O(K)$. Setting the pre-computation of the quadrature constants aside the computational cost in each time step is thus dominated by the computational cost $O(KL)$. Using the ultraspherical spectral method to resolve the spatial aspect of a time-fractional PDE the computational cost of our proposed solver is thus effectively determined by the cost of updating the auxiliary functions using the recursion, which is an efficient process. Coupling a memory cutoff method to the ultraspherical spectral method for comparison can be expected to be at best $O(K N_\ell)$ in each timestep where $N_\ell$ is the amount of time steps which are remembered, cf. \cite[Section 4.2]{jin2023numerical}. Whether the proposed method performs better is thus entirely dependent on the significance of the memory effect of a given equation -- if the number of time steps $N_\ell$ required in memory in the short memory principle method exceed the number of quadrature points $L$ needed to achieve the desired accuracy at time $T$, then the non-classical approach described in this paper will perform better and vice versa. The situation is naturally analogous for the memory footprint. We note however that there are reasons to expect the error behavior of memory cutoff methods to be rather poor in general cases, cf. \cite{ford2001numerical}. Similar observations can be made for the logarithmic memory principle suggested in \cite{ford2001numerical}.
\section{Numerical experiments}\label{sec:numexpsec}
We present four numerical experiments in this section which respectively explore the following subjects:
\begin{enumerate}
\item Recursive nonclassical computation of time-fractional Caputo derivatives.
\item Stability of the recurrence relationship for the auxiliary functions $\psi_j$.
\item A time-fractional heat equation on a triangle.
\item A wave equation dampened by a time-fractional term on a disk.
\end{enumerate}
We note that a worked example which may be considered a numerical experiment was already discussed in Section \ref{sec:recursivesection}. The numerical experiments 3 and 4 were chosen such that the polynomial bases would not simply consist of tensor product domains. Throughout this paper, when discussing error compared to reference solutions we will be considering the pointwise absolute and relative errors 
\begin{align*}
Err_{abs, f}(n,\mathbf{x}) = |\mathbf{P}(\mathbf{x})\bm{f}^n - f(n \Delta t,\mathbf{x})|,\\
Err_{rel,f}(n,\mathbf{x}) = \frac{|\mathbf{P}(\mathbf{x})\bm{f}^n - f(n \Delta t,\mathbf{x})|}{|f(n \Delta t,\mathbf{x})|},
\end{align*}
and indicate which is being considered. We furthermore use Birk--Song's quadrature rule as described in Table 1 throughout this section. As accuracy in the Gaussian quadrature related values $A_j$ and $s_j$ is critical for the method and these values may be precomputed for a given quadrature type without having to ever be re-computed (or one can use lookup tables), we initially use 128-bit floating point numbers and then project them down to 64-bit precision before starting the recurrence loop -- all other values in the numerical experiments are generated and stored as 64-bit floating point numbers.\\
The implementation used for the numerical experiments in this section made use of the open source Julia \cite{beks2017} packages MultivariateOrthogonalPolynomials.jl \cite{noauthor_multivariateorthogonalpolynomialsjl_2021}, FastGaussQuadrature.jl \cite{noauthor_fastgauss_2021} as well as ApproxFun.jl \cite{noauthor_approxfunjl_2021}.
\subsection{Experiment 1: Direct computation of Caputo derivatives}
While constructed for the purpose of solving differential equations, the recursive non-classical approach can nevertheless straightforwardly be used to numerically approximate the Caputo derivative of a known function at a given point. In this section we compare some cases with known explicit Caputo derivatives with numerical approximations obtained via \eqref{eq:caputoODE}. This idea was also explored by Diethelm in \cite{diethelm2008investigation} but his numerical experiments did not include error plots and were limited by small quadrature point counts, a problem which we are unencumbered thanks to \cite{bogaert_iterationfree_2014,glaser_fast_2007,townsend2016fast} and implementations such as FastGaussQuadrature.jl \cite{noauthor_fastgauss_2021}, cf. \cite{townsend2015race}. We make use of the following explicitly given Caputo derivatives for these numerical experiments, cf. \cite[Sec. 3.1]{herrmann_fractional_2014}:
\begin{align}\label{eq:numexp1polynomial}
\frac{\partial^\alpha}{\partial t^\alpha}(t^2) &= \frac{2t^{2-\alpha}}{\Gamma(3-\alpha)}, \quad \alpha \in [0,1],\\
\frac{\partial^{\nicefrac{1}{2}}}{\partial t^{\nicefrac{1}{2}}}E_{a,1}(t) &= -\frac{1}{\sqrt{t \pi}} + \underset{k=0}{\overset{\infty }{\sum }}\frac{k! t^{k-\frac{1}{2}}}{\Gamma \left(k+\frac{1}{2}\right) \Gamma (1+a k)}, \quad a>0,\label{eq:numexp2mittag}
\end{align}
where $E_{a,b}(z)$ denotes the two-parameter Mittag-Leffler function \cite{mittag1903nouvelle,wiman1905fundamentalsatz}, which is defined by (cf. \cite[10.46.3]{nist_2022}):
\begin{align*}
E_{\alpha, \beta} (z) := \sum_{k=0}^\infty \frac{z^k}{\Gamma(\alpha k + \beta)}.
\end{align*}
We furthermore have (cf. \cite[Chapter 4]{diethelm_analysis_2010})
\begin{align*}
E_{1, 1}(z) &= e^z,
\end{align*}
which via \eqref{eq:numexp2mittag} yields
\begin{align}\label{eq:numexp3exp}
\frac{\partial^{\nicefrac{1}{2}}}{\partial t^{\nicefrac{1}{2}}}e^t = e^t \text{erf}\left(\sqrt{t}\right).
\end{align}
In Figures \ref{fig:rawcaputoerror} and \ref{fig:rawcaputoerror2} we show error plots comparing recursively computed numerical approximations to the explicitly known solutions in \eqref{eq:numexp1polynomial} and \eqref{eq:numexp3exp} with varying time step sizes $\Delta t$ and quadrature nodes $L$ respectively. Note that as suggested by the error bounds in Section \ref{sec:error} we observe linear improvement in accuracy in $\Delta t$ once sufficient convergence in $L$ has been achieved.
\begin{figure}\centering
     \subfloat[$f(t) = t^2$, $L=65$, $\alpha = \frac{2}{3}$]
    {{ \centering \includegraphics[width=9cm]{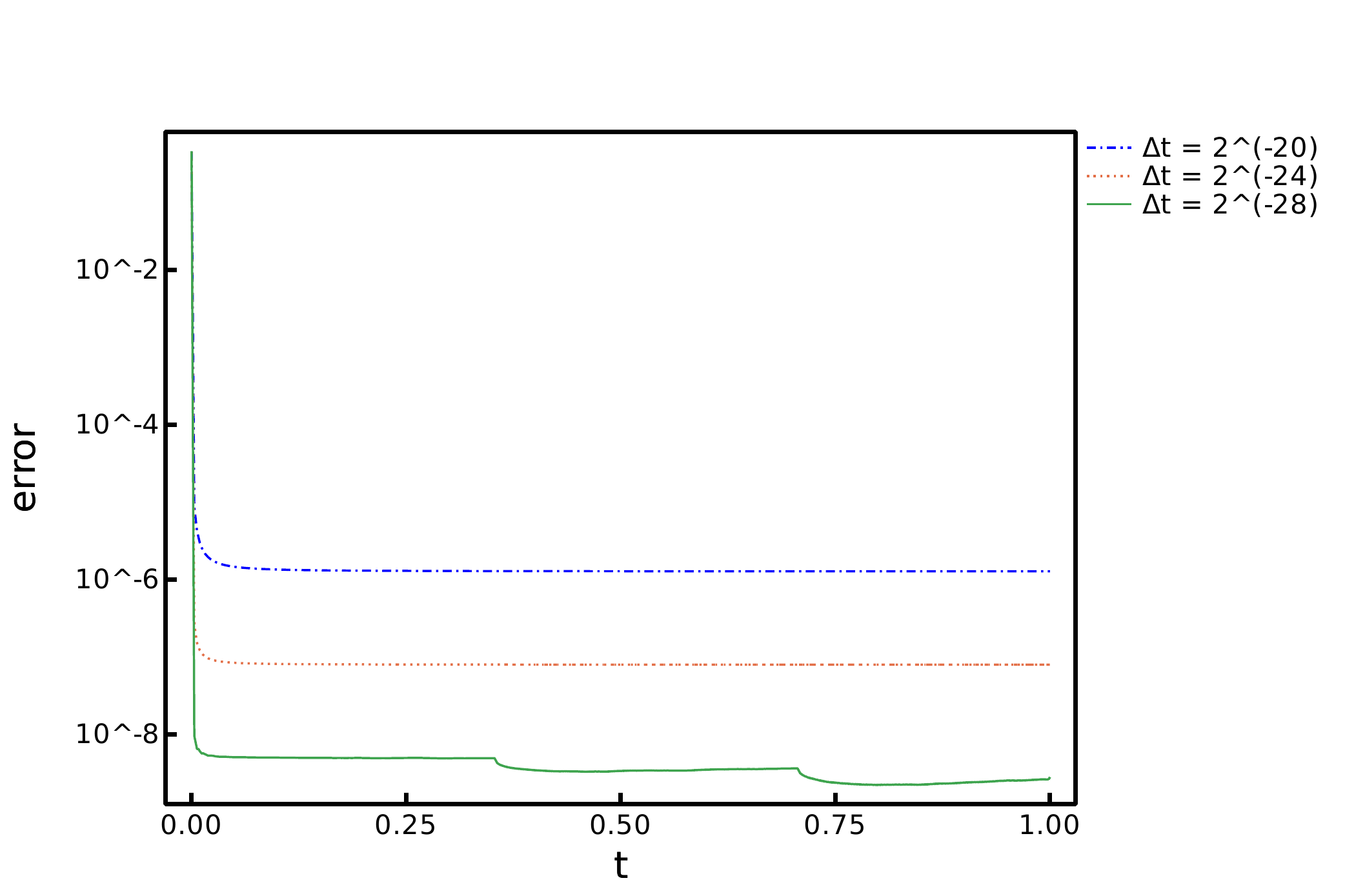} }}\\
     \subfloat[$f(t) = e^t$, $L=60$, $\alpha = \frac{1}{2}$]
    {{ \centering \includegraphics[width=9cm]{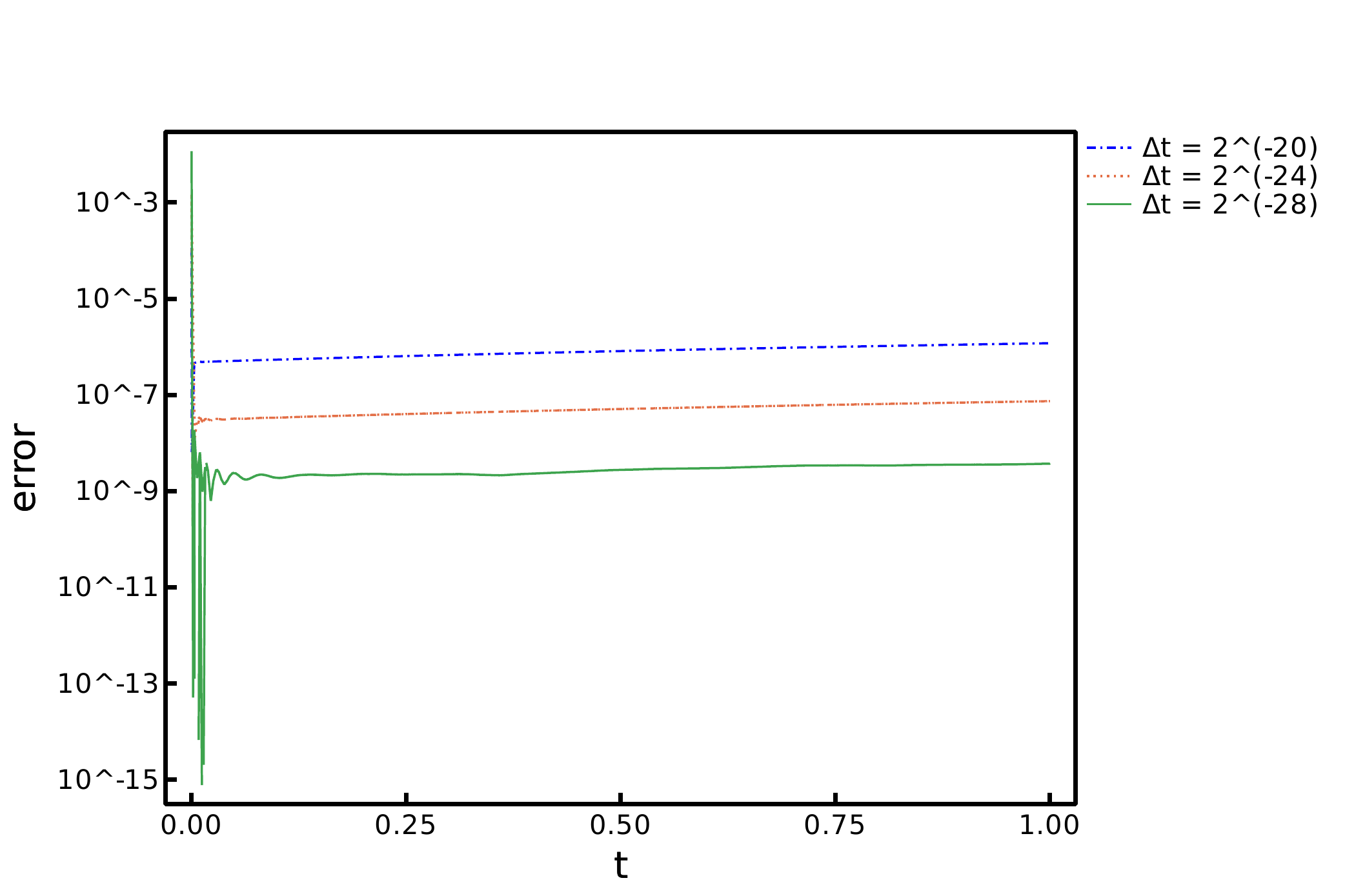} }}
    \caption{Relative errors between analytic and recursively computed Caputo derivatives of indicated $f(t)$ with number of quadrature points $L$ for various time step sizes $\Delta t$.}
    \label{fig:rawcaputoerror}
    \end{figure}
    \begin{figure}\centering
     \subfloat[$f(t) = t^2$, $\alpha = \frac{1}{4}$]
    {{ \centering \includegraphics[width=9cm]{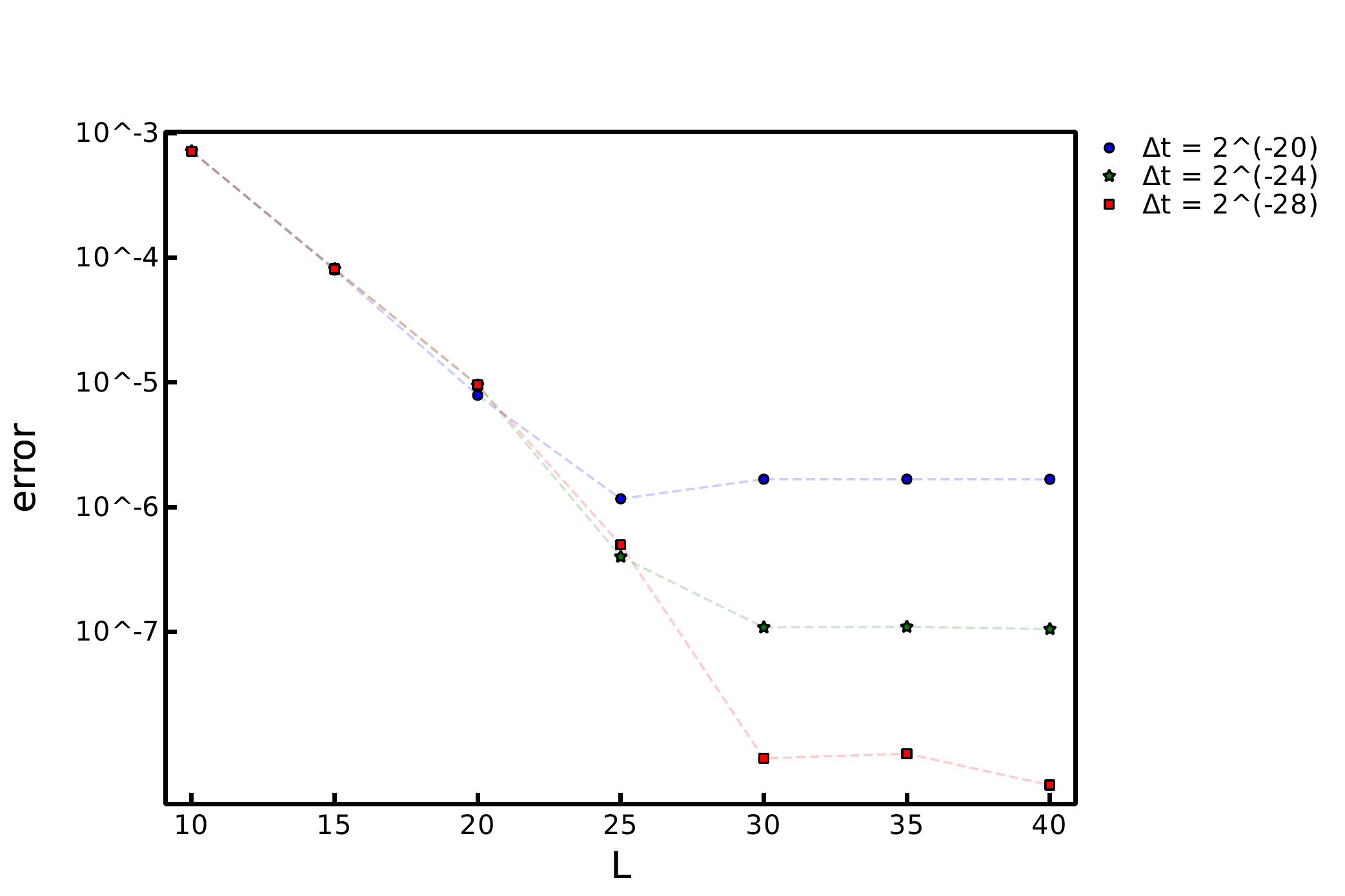} }}\\
     \subfloat[$f(t) = e^t$, $\alpha = \frac{1}{2}$]
    {{ \centering \includegraphics[width=9cm]{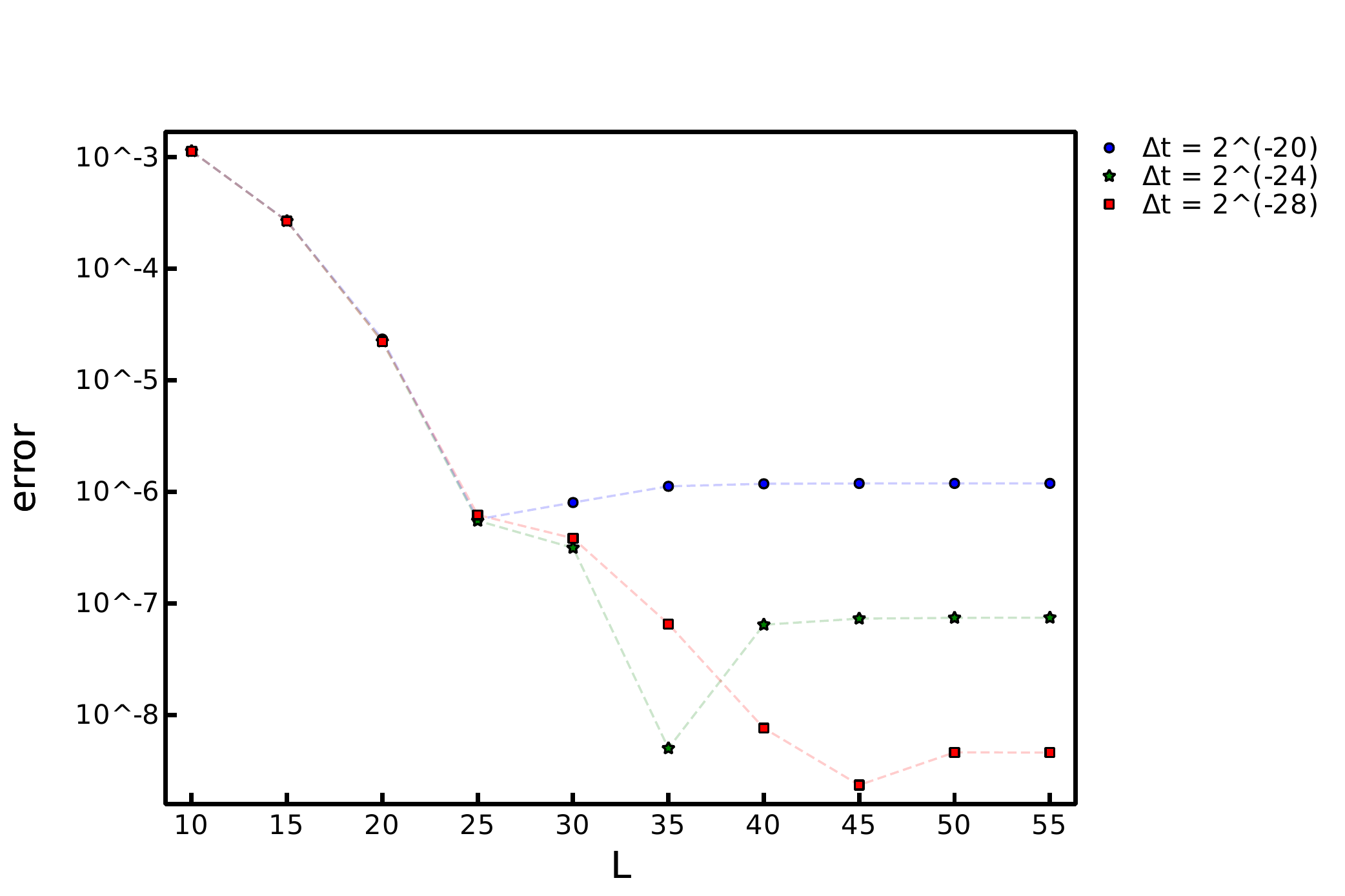} }}
    \caption{Relative errors between analytic and recursively computed Caputo derivatives of indicated $f(t)$ at the end point of the time domain $t \in [0,1]$ as a function of the number of quadrature points $L$ used with various time step sizes $\Delta t$.}
    \label{fig:rawcaputoerror2}
    \end{figure}
\subsection{Experiment 2: Stability of the auxiliary function recurrence}\label{sec:numexppsi}
A remaining concern with the proposed method may be that floating point arithmetic can cause recurrence relationships to be unstable. In this section, we thus provide some numerical experiments in which we compare full domain numerical integration in \eqref{eq:auxiliarydef} to the recurrence approach in \eqref{eq:caputoPDEpsi} to see how this error behaves in practice, using the explicitly known solutions of the previous numerical experiment as a template.\\
Figures \ref{fig:psistability1} and \ref{fig:psistability2} show the error between the recurrence and full domain approaches to computing the auxiliary functions $\psi_j$. Consistent with the error analysis in Section \ref{sec:error}, we observe that if $L$ is low, reducing $\Delta t$ does not improve the error. As suggested by the error bounds discussion in Section \ref{sec:error}, one should thus increase $L$ to convergence before excessively refining the step size $\Delta t$.
\begin{figure}\centering
    {{ \centering \includegraphics[width=10cm]{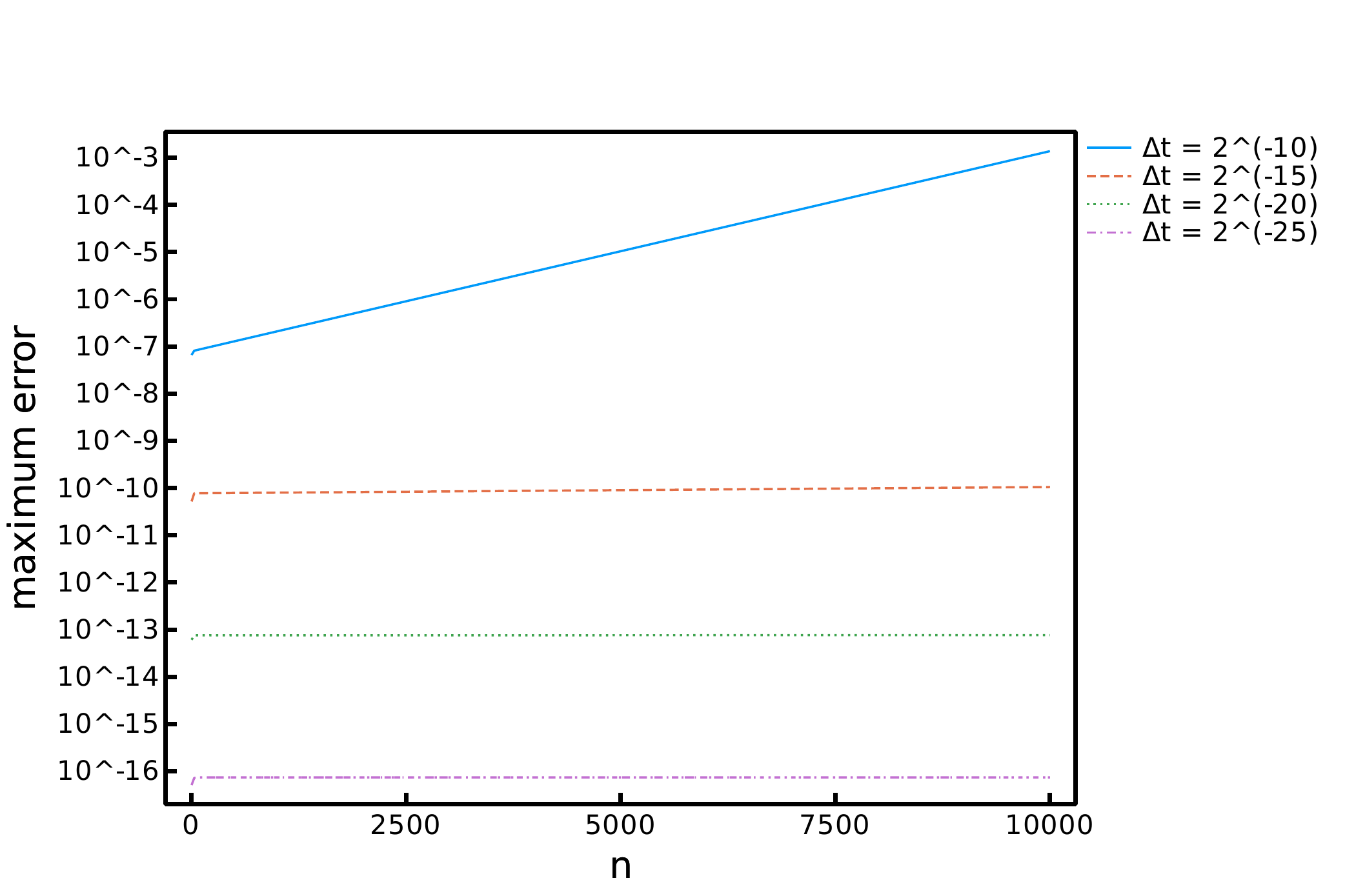} }}
    \caption{Maximal-in-$j$ absolute errors at each time step between auxiliary functions $\psi_j$ after $n$ time steps of indicated size $\Delta t$ computed via the recurrence and full domain quadrature for the problem in \eqref{eq:numexp3exp} with $L = 30$ quadrature points using the Birk--Song quadrature rule. For a high $n$ endpoint experiment see Figure \ref{fig:psistability2}.}
    \label{fig:psistability1}
    \end{figure}
\begin{figure}\centering
    {{ \centering \includegraphics[width=8.5cm]{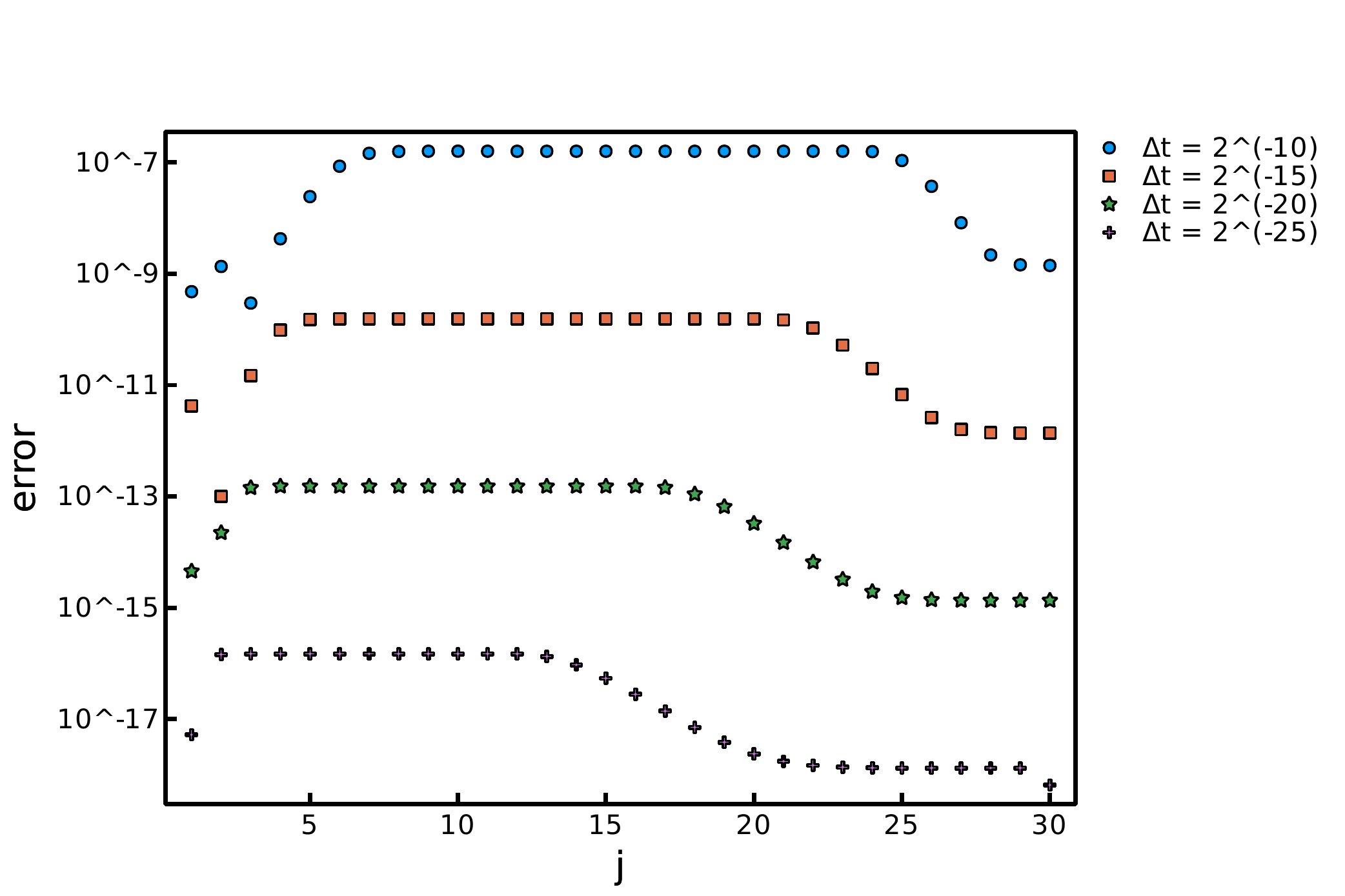} }}
    \caption{Absolute errors between auxiliary functions $\psi_j$, with $j$ on the $x$-axis, after 10 million time steps of indicated size $\Delta t$ computed via the recurrence and full domain approach for the problem in \eqref{eq:numexp1polynomial} with $L = 30$ quadrature points and $\alpha = \frac{1}{2}$ using the Birk--Song quadrature rule. For an error over time experiment see Figure \ref{fig:psistability1}.}
    \label{fig:psistability2}
    \end{figure}
\subsection{Experiment 3: A time-fractional heat equation on the triangle}
In this section we use the Proriol polynomials $\mathbf{P}^{(a, b, c)} (x,y)$ which are orthogonal on the triangle domain $$\Omega_T = \left\{ (x,y) : 0 \leq x, 0 \leq y \leq 1-x  \right\}$$ with respect to the weight function $w(x,y) = x^a y^b (1-x-y)^c$, cf. \cite{olver_sparse_2019,dunkl_orthogonal_2014}, to solve the time-fractional heat equation
\begin{align}\label{eq:trianglenumexp}
\frac{\partial^\alpha}{\partial t^\alpha} f - k \Delta f = 0,\\
f(0,x,y) = f_0(x,y),
\end{align}
with zero Dirichlet conditions on the boundary of $\Omega_T$. While non-zero Dirichlet boundary conditions would require the use of restriction operators \cite{olver_sparse_2019}, zero Dirichlet conditions can be enforced by resolving $f(x,y)$ in the weighted Proriol polynomial basis $$\mathbf{W}^{(1,1,1)}(x,y) = xy(1-x-y)\mathbf{P}^{(1,1,1)} (x,y).$$ The Laplacian $\Delta$ then maps $\mathbf{W}^{(1,1,1)}(x)$ to $\mathbf{P}^{(1,1,1)}(x)$ and a sparse conversion operator $\mathcal{C}$ for this basis map is also available, see Figure \ref{fig:spytriangle} for sparsity pattern plots. We thus obtain the recursive non-classical method
\begin{align}\label{eq:trianglemethod}
&\mathbf{P}^{(1,1,1)}(x,y) \left( - k\Delta +  \left( \sum_{j=1}^L A_j \frac{(1-e^{-s_j^2 \Delta t})}{s_j^2 \Delta t} \right)\mathcal{C}  \right) \bm{f}^n \\& \quad =  \mathbf{P}^{(1,1,1)}(x,y) \mathcal{C}  \left(\left( \sum_{j=1}^L A_j \frac{(1-e^{-s_j^2 \Delta t})}{s_j^2 \Delta t} \right)\bm{f}^{n-1} - \left( \sum_{j=1}^L A_j e^{-s_j^2\Delta t} \bm{\psi}_j^{n-1}\right) \right),\nonumber
\end{align}
valid for the heat equation with zero Dirichlet boundary conditions on the triangle domain $\Omega_T$ and $k>0$. Note that the solution obtained via the above scheme is $$f(N\Delta t,x,y) \approx \mathbf{W}^{(1,1,1)}(x,y) \bm{f}^N.$$ In Figure \ref{fig:heattriangleexample} we plot numerical solutions obtained for the time-fractional heat equation on the triangle with $k=\frac{1}{500}$ with fractional order $\alpha = \frac{3}{5}$ and initial condition
\begin{align}\label{eq:initialconditiontriangle}
f_0(x,y) =  100xy(1-x-y)^2e^{-10\left((x-0.6)^2+(y-0.2)^2\right)}. 
\end{align}
Figure \ref{fig:trianglecoeffs} shows the corresponding coefficient decay in the polynomial approximation which can be used as a proxy for whether the desired spatial convergence of a given application has been achieved, cf. \cite{olver_sparse_2019}. Stepping through $1000$ time steps with $L = 45$ and $K = 400$ takes roughly 143 milliseconds on a Lenovo ThinkPad X1 Carbon laptop with a 12th Gen Intel i5-1250P CPU.
    \begin{figure}\centering
     \subfloat[$\Delta$ operator.]
    {{ \centering \includegraphics[width=6.4cm]{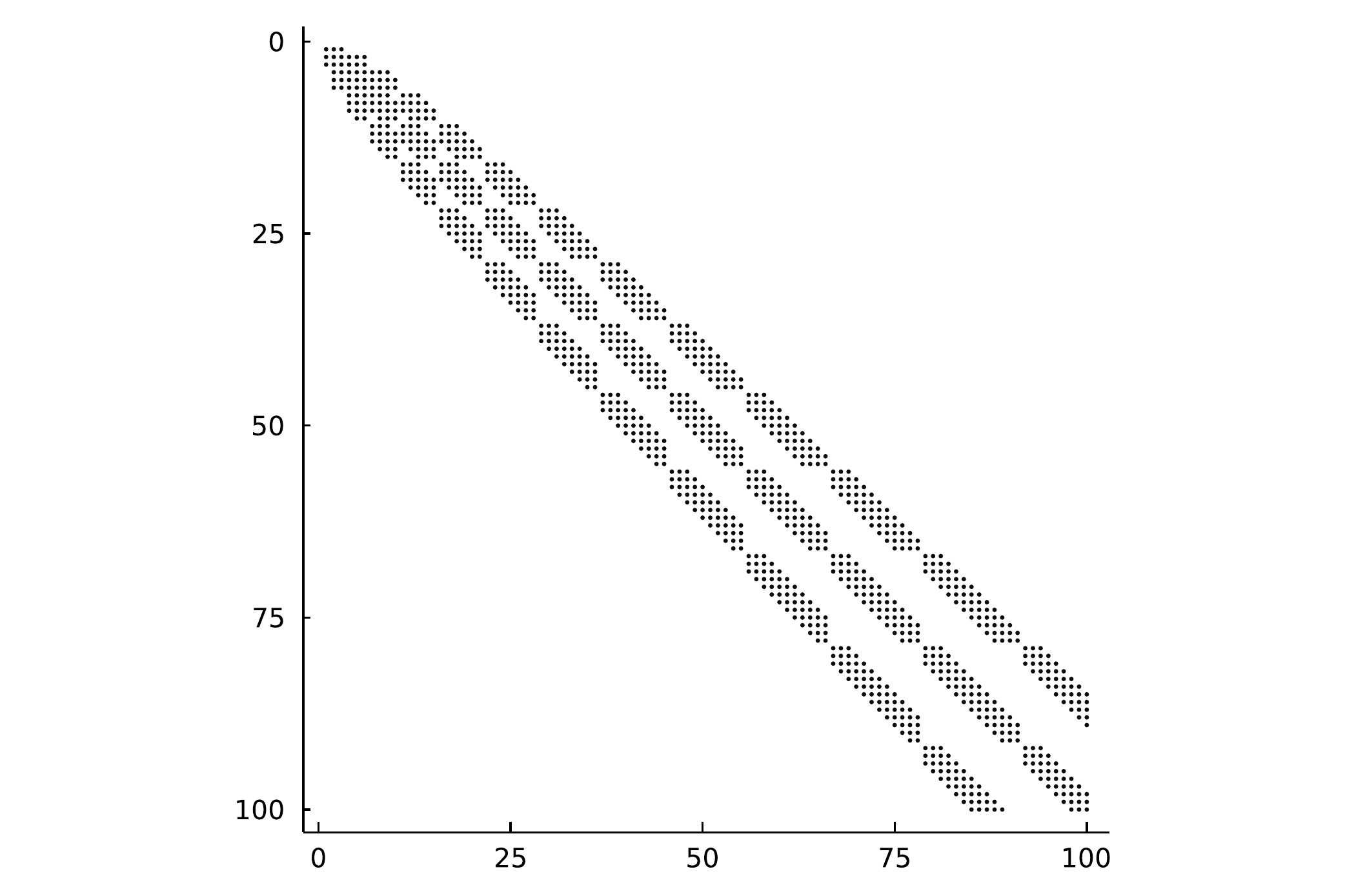} }}
     \subfloat[$\mathcal{C}$ operator.]
    {{ \centering \includegraphics[width=6.4cm]{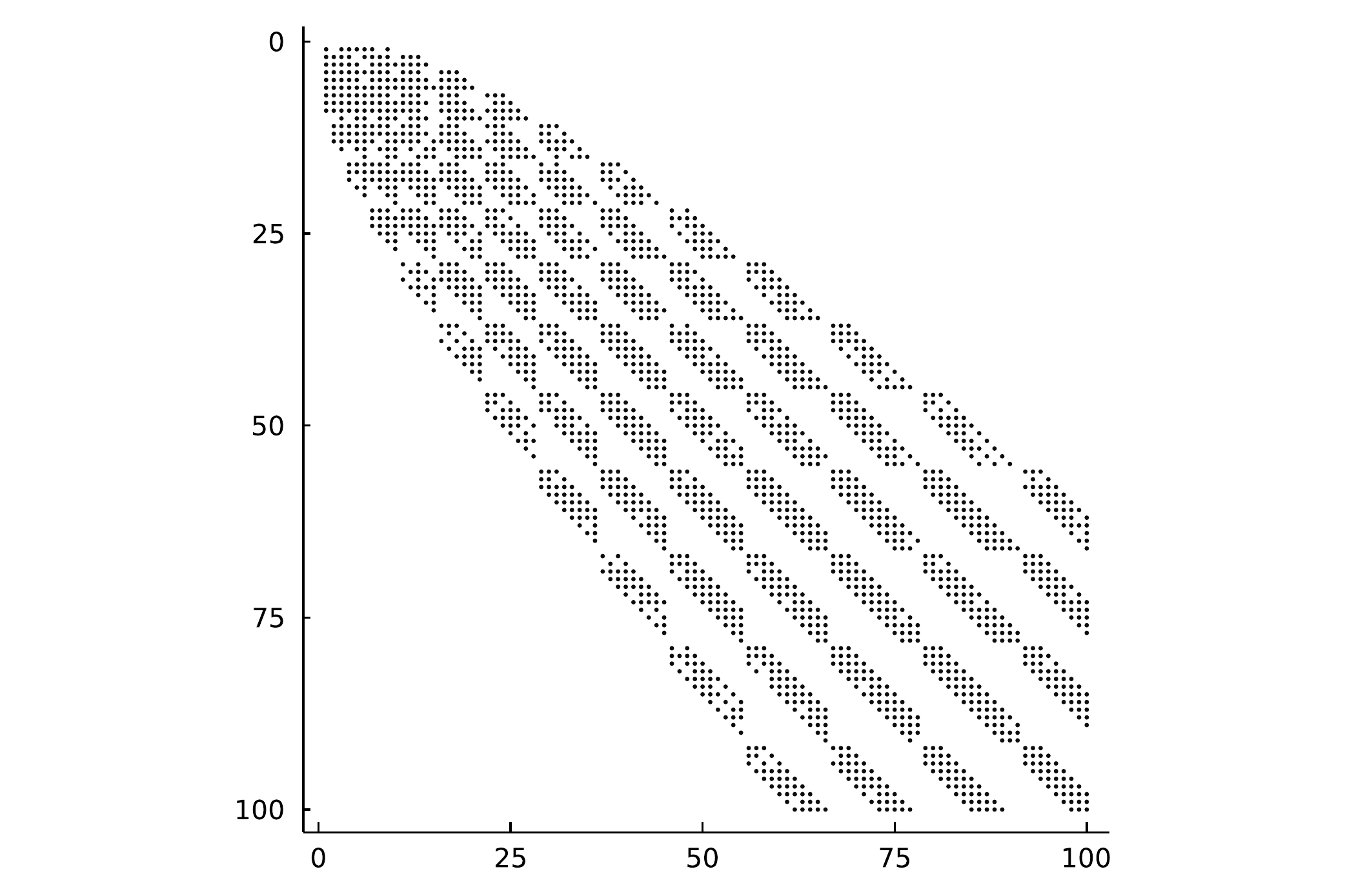} }}
    \caption{Sparsity pattern (spy) plots for $100\times 100$ blocks of the Laplacian $\Delta$ and conversion operator $\mathcal{C}$ appearing in Equation \eqref{eq:trianglemethod}. See \cite{olver_sparse_2019} for how to compute their entries.}
    \label{fig:spytriangle}
    \end{figure}
        \begin{figure}\centering
     \subfloat[$t = 0$]
    {{ \centering \includegraphics[width=6.4cm]{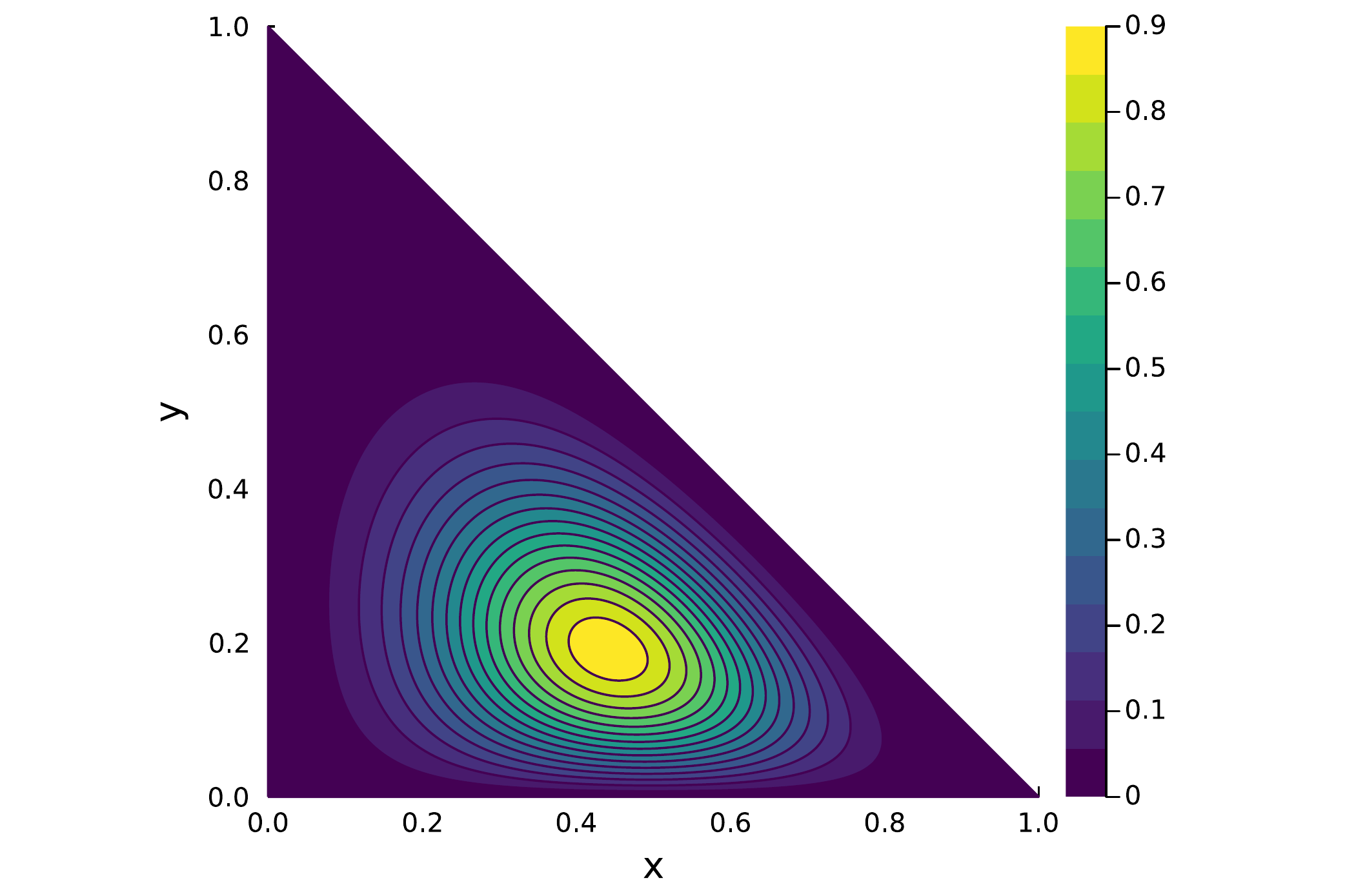} }}
     \subfloat[$t = \frac{1}{3}$]
    {{ \centering \includegraphics[width=6.4cm]{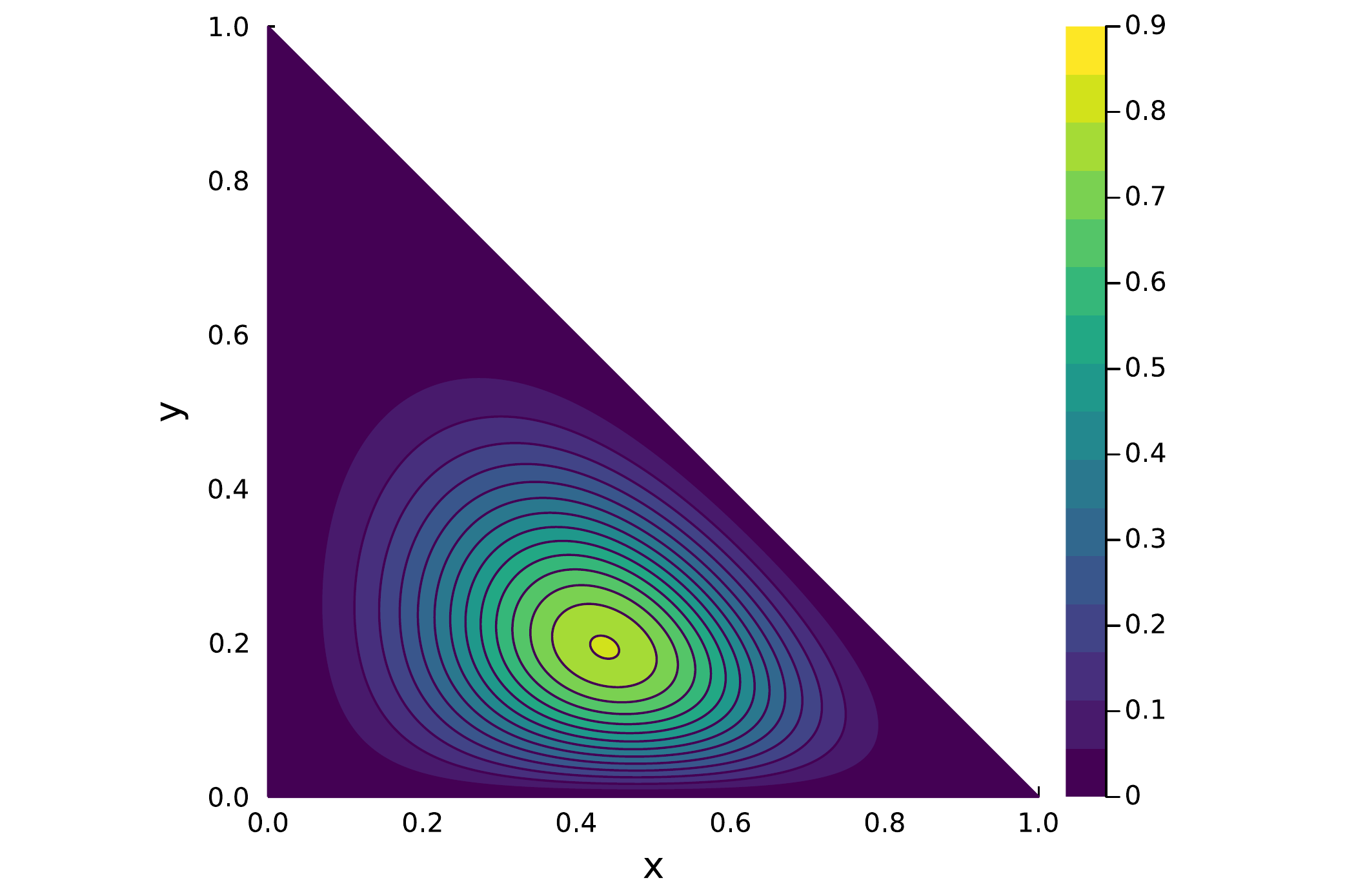} }}\\
         \subfloat[$t = \frac{2}{3}$]
    {{ \centering \includegraphics[width=6.4cm]{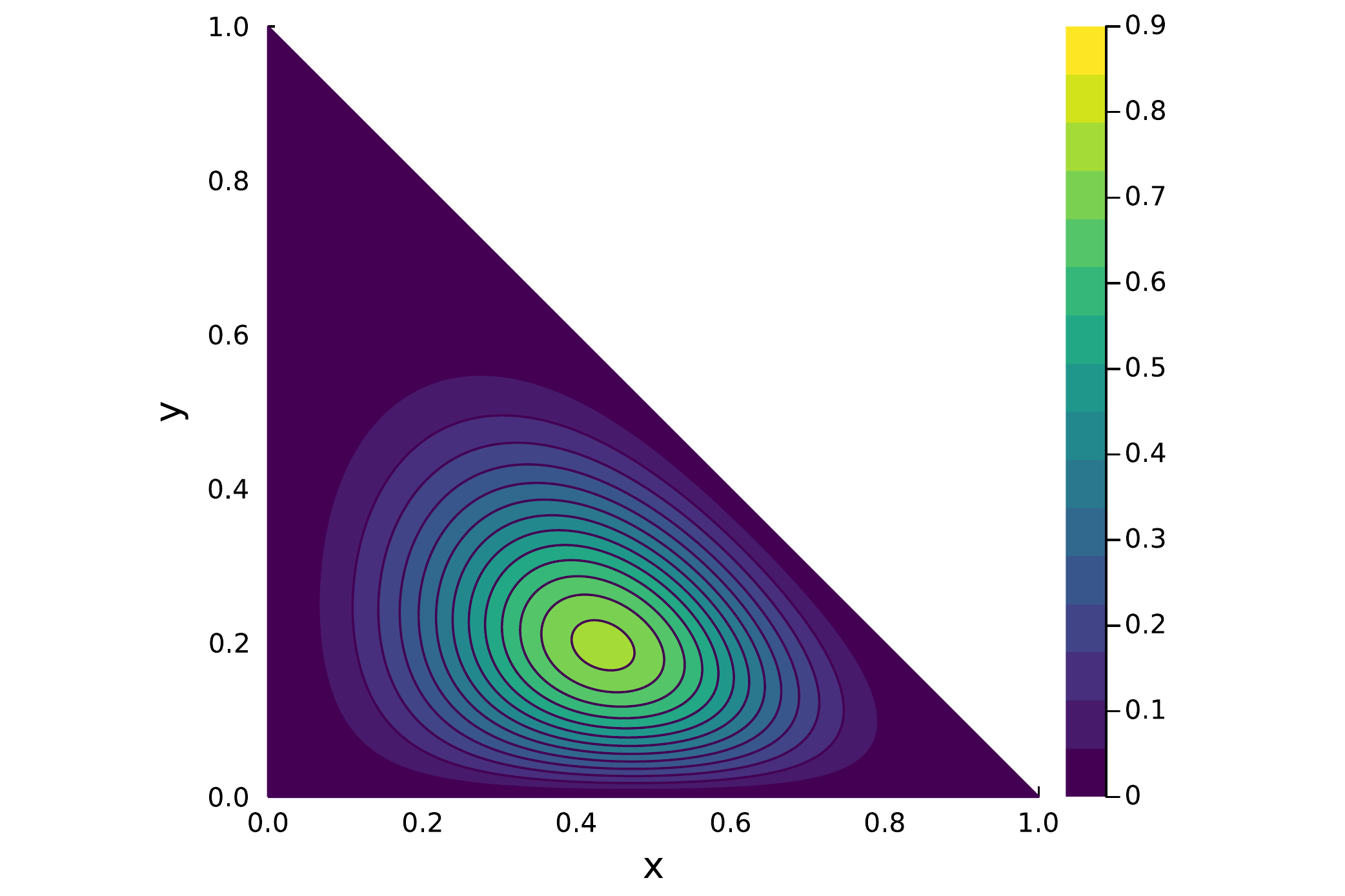} }}
     \subfloat[$t = 1$]
    {{ \centering \includegraphics[width=6.4cm]{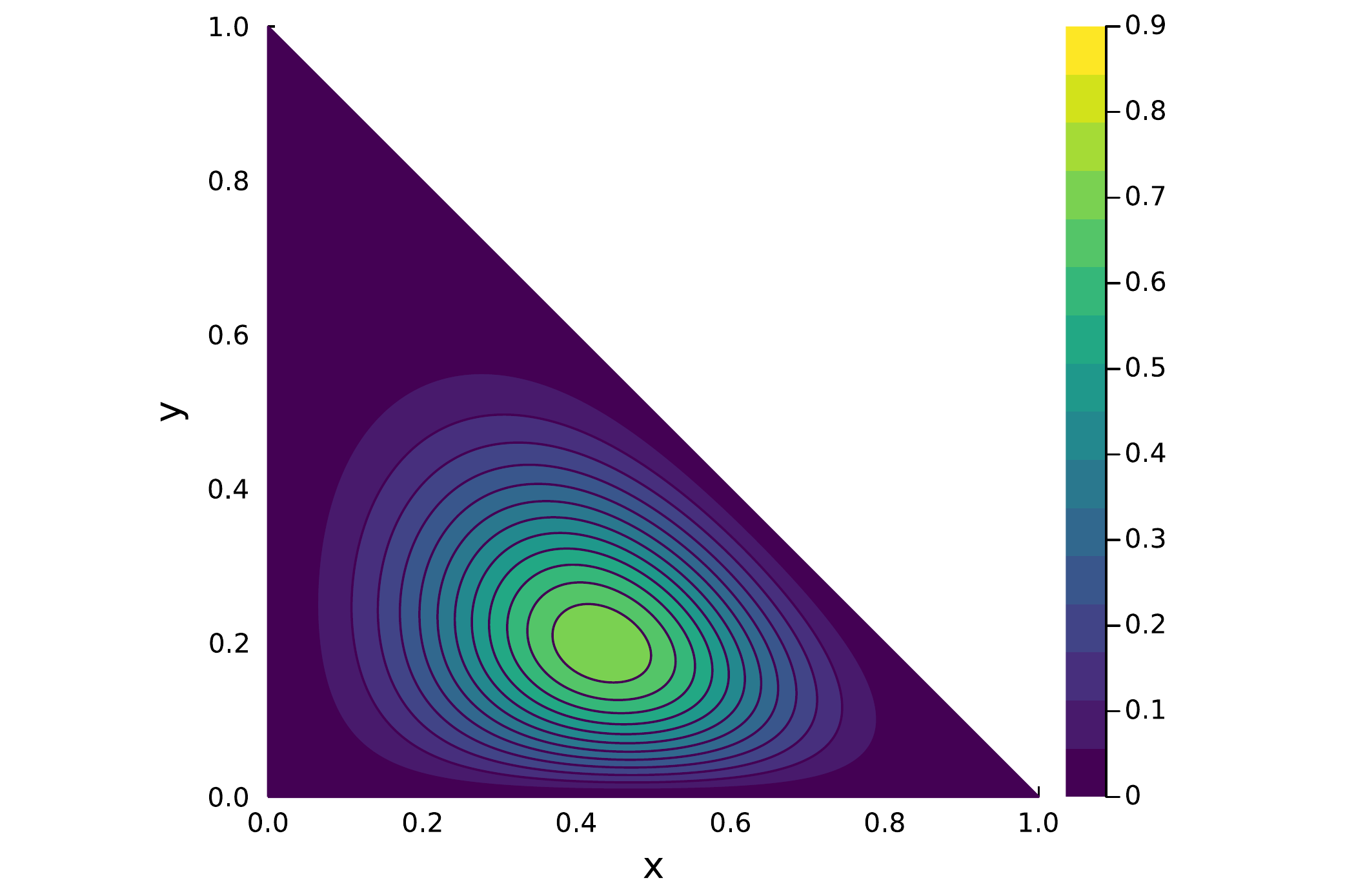} }}
    \caption{(b), (c) and (d) show numerical solutions at different times $t$ to the problem in Eq. \eqref{eq:trianglenumexp} with $\alpha = \frac{3}{5}$ and $k=\frac{1}{500}$, starting from the initial condition in Eq. \eqref{eq:initialconditiontriangle} which is shown in (a). The num. solutions are degree $K=400$ approximations computed with $L=45$ quadrature points and $\Delta t = 2^{-20}$. To make comparison easier the scale on all plots is kept consistent.}
        \label{fig:heattriangleexample}
    \end{figure}
\begin{figure}\centering
    {{ \centering \includegraphics[width=8.5cm]{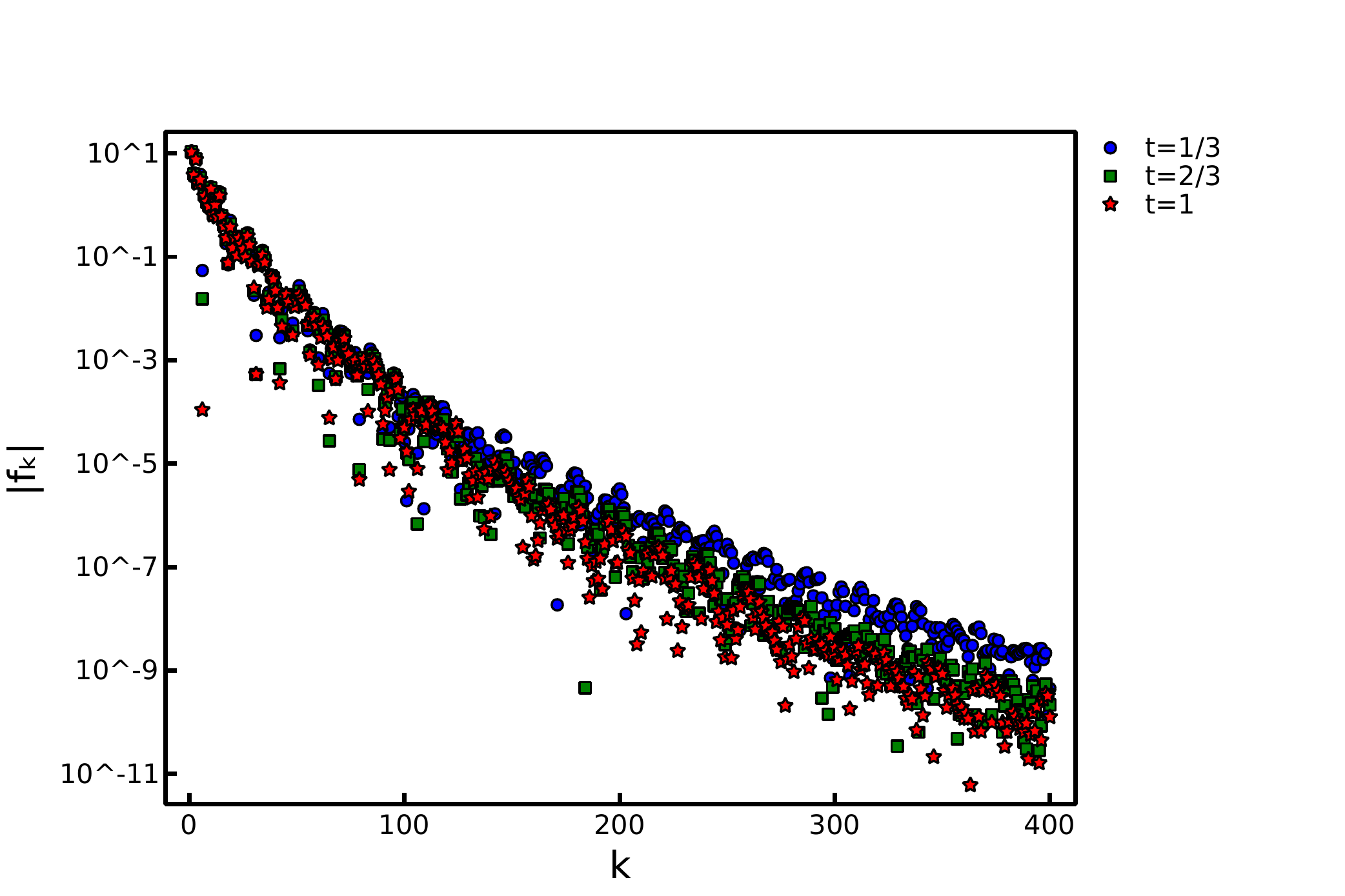} }}
    \caption{Semilogarithmic plot of absolute values of the $k$-th coefficients of the weighted Proriol polynomial approximations corresponding to the numerical solutions shown in Figure \ref{fig:heattriangleexample}.}
    \label{fig:trianglecoeffs}
    \end{figure}

\subsection{Experiment 4: A wave equation dampened by a fractional term on the disk}\label{sec:treebynumexp}
We consider a problem closely associated with the motivating model used in \cite{treeby2010modeling,treeby2014modeling} to describe medical ultrasound devices used on the human brain. As mentioned in the introduction, the equation of interest (cf. \cite{treeby2014modeling}) is
\begin{align*}
\frac{1}{c_0} \frac{\partial^2}{\partial t^2} f - \Delta f - \tau \frac{\partial^\alpha}{\partial t^\alpha} (\Delta f) = 0,
\end{align*}
posed on the three dimensional ball surrounded by multiple matching spherical shells of varying thickness, cf. the sphere head models described in e.g. \cite{naess2017corrected}. The presence of the term $\frac{\partial^\alpha}{\partial t ^\alpha} (\Delta f)$ as well as the requirement to decompose the domain into matching segments add complications to this problem which we intend to address in a separate paper. In this section we discuss the similar but slightly simplified equation
\begin{align}\label{eq:standinproblem}
\frac{1}{c_0^2}\frac{\partial^2}{\partial t^2} f - \Delta f + \tau\frac{\partial^\alpha}{\partial t^\alpha} f = 0,
\end{align}
posed on a two-dimensional disk with zero Dirichlet conditions and given initial conditions $f(0,\mathbf{x})$ and $\left(\frac{\partial}{\partial t}f\right)(0,\mathbf{x})$. Note that this stand-in equation also describes power law attenuation in viscous media -- it is in fact a modified version of Szabo’s wave equation, see \cite{szabo1994time,chen2003modified} which has also been suggested for use in applications, cf. \cite{holm2016wave}. To our knowledge no non-trivial explicit solutions for this equation are known.\\
As the problem is posed on a disk domain, we use \emph{generalized} Zernike polynomials $\mathbf{Z}^{(b)}(r,\theta)$. Zernike polynomials have a longstanding history \cite{zernike1935hyperspharische,zernike_beugungstheorie_1934} not just in mathematics research but also in natural science and engineering applications, especially in the field of optics \cite{thibos_standards_2000, roddier_atmospheric_1990,mahajan_zernike_1994,rocha_effects_2007,noll_zernike_1976}, and as result a handful of different and conflicting conventions exist for them. The convention we follow for the generalized Zernike polynomials is that $\mathbf{Z}^{(b)}(r,\theta)$ are orthonormal polynomials on the unit disk with respect to the weight function $w(r) = (1-r^2)^b$, whose radial components are given by orthonormal Jacobi polynomials $\tilde{\mathbf{P}}^{(\alpha, \beta)}(2r^2-1)$, i.e.:
\begin{align*}
Z_{\ell,m}^{(b)}(r,\theta) =  2^{\frac{\abs{m}+a+b+2}{2}} r^{\abs{m}}\tilde{P}_{\frac{\ell-\abs{m}}{2}}^{(b,\abs{m})}(2r^2-1) \sqrt{\tfrac{2-\delta_{m,0}}{2\pi}} \left\{\begin{array}{ccc} \cos(m\theta) & {\rm for} & m \ge 0,\\ \sin(\abs{m}\theta) & {\rm for} & m < 0.\end{array}\right.
\end{align*}
The Zernike polynomials are traditionally given in terms of polar coordinates but they are polynomials in $x$ and $y$ and generally not polynomials in $r$ and $\theta$. For a discussion of sparse spectral methods using Zernike polynomials, see e.g. \cite{vasil_tensor_2016}. For open source numerical implementations we refer to FastTransforms \cite{noauthor_fasttransforms_2021,noauthor_fasttransformsjl_2021} as well as MultivariateOrthogonalPolynomials.jl \cite{noauthor_multivariateorthogonalpolynomialsjl_2021}.\\
As in the triangle example, zero Dirichlet boundary conditions can be enforced by choosing an appropriate weighted basis $$\mathbf{W}^{(b)}(r,\theta) = (1-r^2)^b \mathbf{Z}^{(b)}(r,\theta),$$ to resolve the solution function at each time step $f(n\Delta t,r,\theta)$. Using a simple backward finite difference approximation for the second derivative we then find the following scheme for Equation \eqref{eq:standinproblem}:
\begin{align}\label{eq:diskmethod}
&\mathbf{Z}^{(1)}(r,\theta) \left(\left(\tfrac{1}{c_0^2 \Delta t} + \tau \left( \sum_{j=1}^L A_j \tfrac{1-e^{-s_j^2\Delta t}}{s_j^2} \right)\right) \mathcal{C} - (\Delta t) \Delta \right) \bm{f}^n = \\ &\mathbf{Z}^{(1)}(r,\theta) \mathcal{C}\left( \left( \tfrac{2}{c_0^2 \Delta t} + \tau \sum_{j=1}^L A_j \tfrac{1-e^{-s_j^2\Delta t}}{s_j^2}  \right) \bm{f}^{n-1} - \tfrac{\bm{f}^{n-2}}{c_0^2 \Delta t} - \tau \sum_{j=1}^L A_j e^{-s_j^2\Delta t} \bm{\psi}_j^{n-1} \right),\nonumber
\end{align}
with approximate solution given by $$f(N\Delta t, r, \theta) = \mathbf{W}^{(1)}(r,\theta) \bm{f}^N.$$ Higher order schemes may be obtained in the straightforward way. As in the triangle experiment, we plot a sparsity pattern plot for $\mathcal{C}$ in Figure \eqref{fig:spydisk} but omit the spy plot for the Laplacian $\Delta$ on these bases since it is in fact simply diagonal when mapping from $\mathbf{W}^{(1)}(r,\theta)$ to $\mathbf{Z}^{(1)}(r,\theta)$.
\begin{remark}
The second derivative in $t$ requires access to at least two previous points in time, one of which can be re-used in the Caputo part of the equation. As a result, while the minimum total memory requirement discussed in Section \ref{sec:memorycost} still holds, the memory \emph{increase} from adding a Caputo term to a wave equation is $K$ less than said minimum since $\bm{f}^{n-1}$ is a vector of length $K$ and need not be stored twice.
\end{remark}
We plot a numerical solution to Equation \eqref{eq:standinproblem} with parameters $\tau = 1$, $c = 100$, $\alpha = \frac{1}{2}$ and initial conditions
\begin{align}\label{eq:diskinitialnumexp}
f(0,x,y) &= 4y(1-x^2-y^2)^2,\\
\frac{\partial f}{\partial t}(0,x,y) &= 0.\nonumber
\end{align}
on the disk in Figure \ref{fig:diskexample1surface}. Four animations of solutions to this problem for various parameters and initial conditions, including the one pictured in Figure \ref{fig:diskexample1surface}, have been made available via FigShare, see \cite{gutleb_wave_2023}.\\
In Figure \ref{fig:diskimagingexample} we plot an example of an initial condition along with a circular array of 70 sensors and the data these sensors read out over time. State-of-the-art software such as the Matlab package k-wave, cf. \cite{treeby2010k,treeby2012modeling,treeby2018rapid,treeby2014modelling}, can use data from such sensor arrays to approximately reconstruct source images with natural applications in medical ultrasound as well as photo-acoustic imaging. This section serves as a proof of concept that a recursive sparse spectral element method may be applicable to such problems in the future -- we discuss additional research required to reach that point in the next section.

    \begin{figure}\centering
    \centering \includegraphics[width=8.5cm]{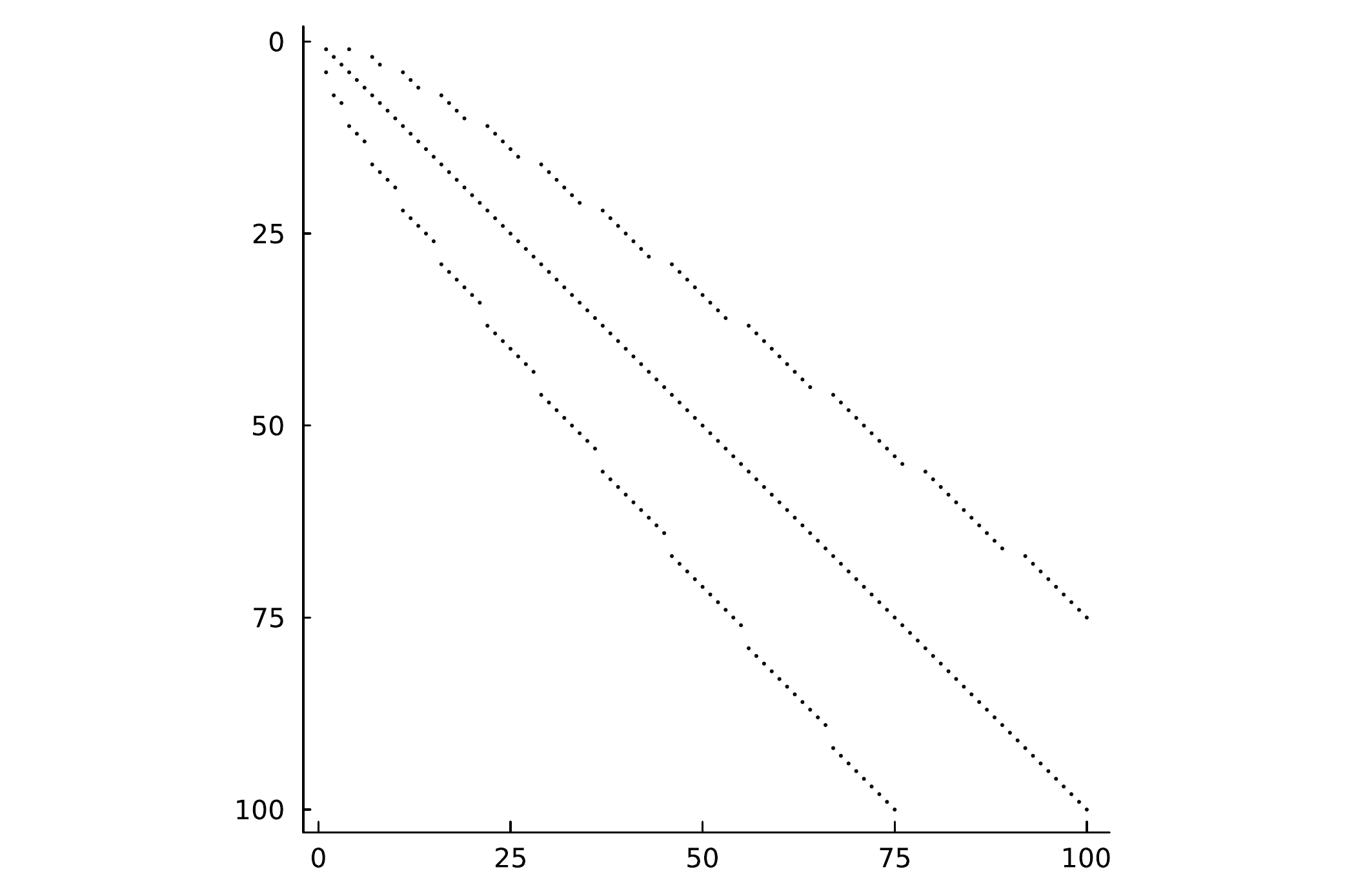}
    \caption{Sparsity pattern (spy) plot for $100\times 100$ block of the conversion operator $\mathcal{C}$ appearing in Equation \eqref{eq:diskmethod}. See e.g. \cite{vasil_tensor_2016} for how to compute its entries.}
    \label{fig:spydisk}
    \end{figure}
    \begin{figure}\centering
     \subfloat[$T = 0$]
    {{ \centering \includegraphics[width=7cm]{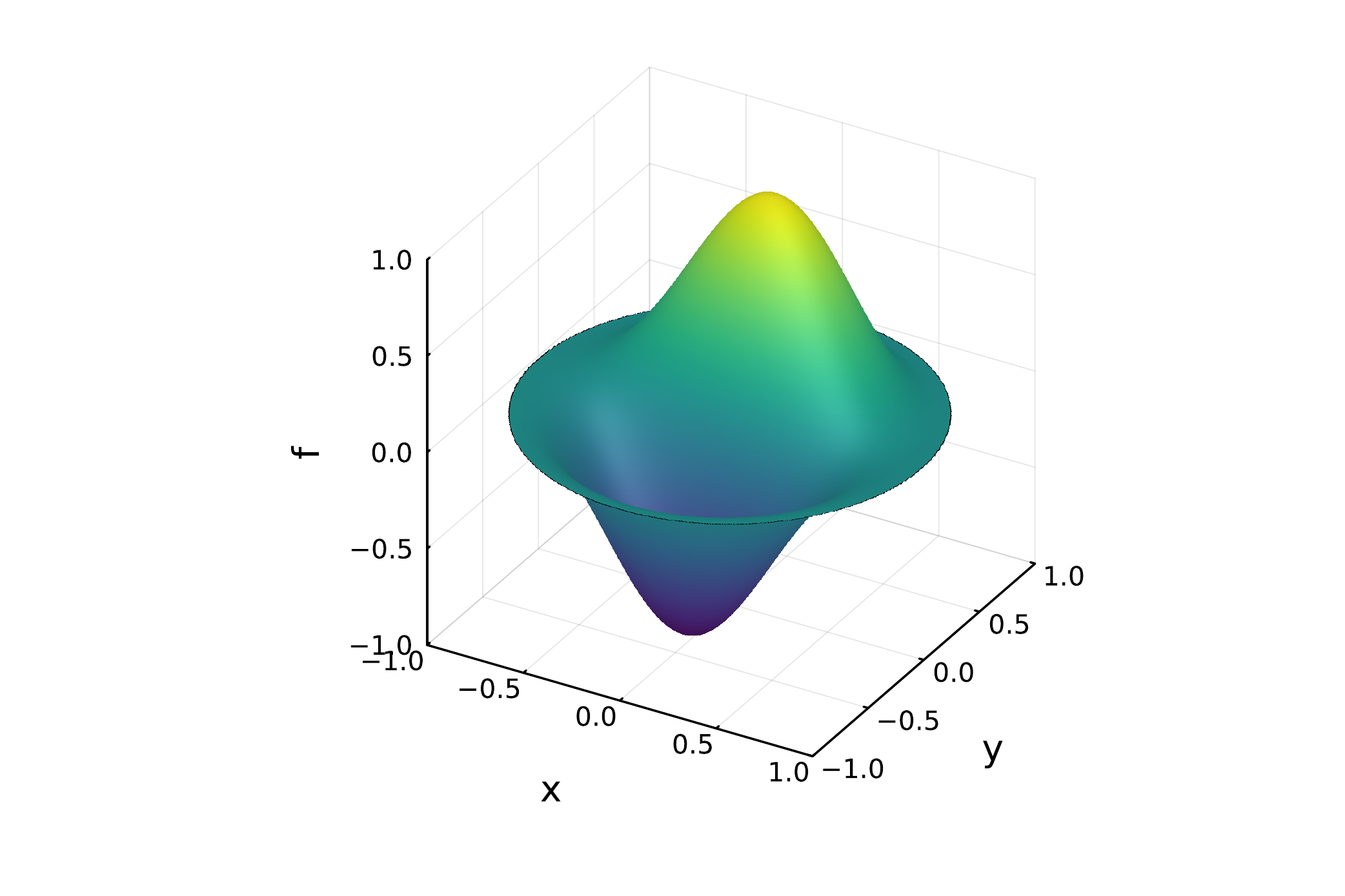} }}
     \subfloat[$T = 19200$]
    {{ \centering \includegraphics[width=7cm]{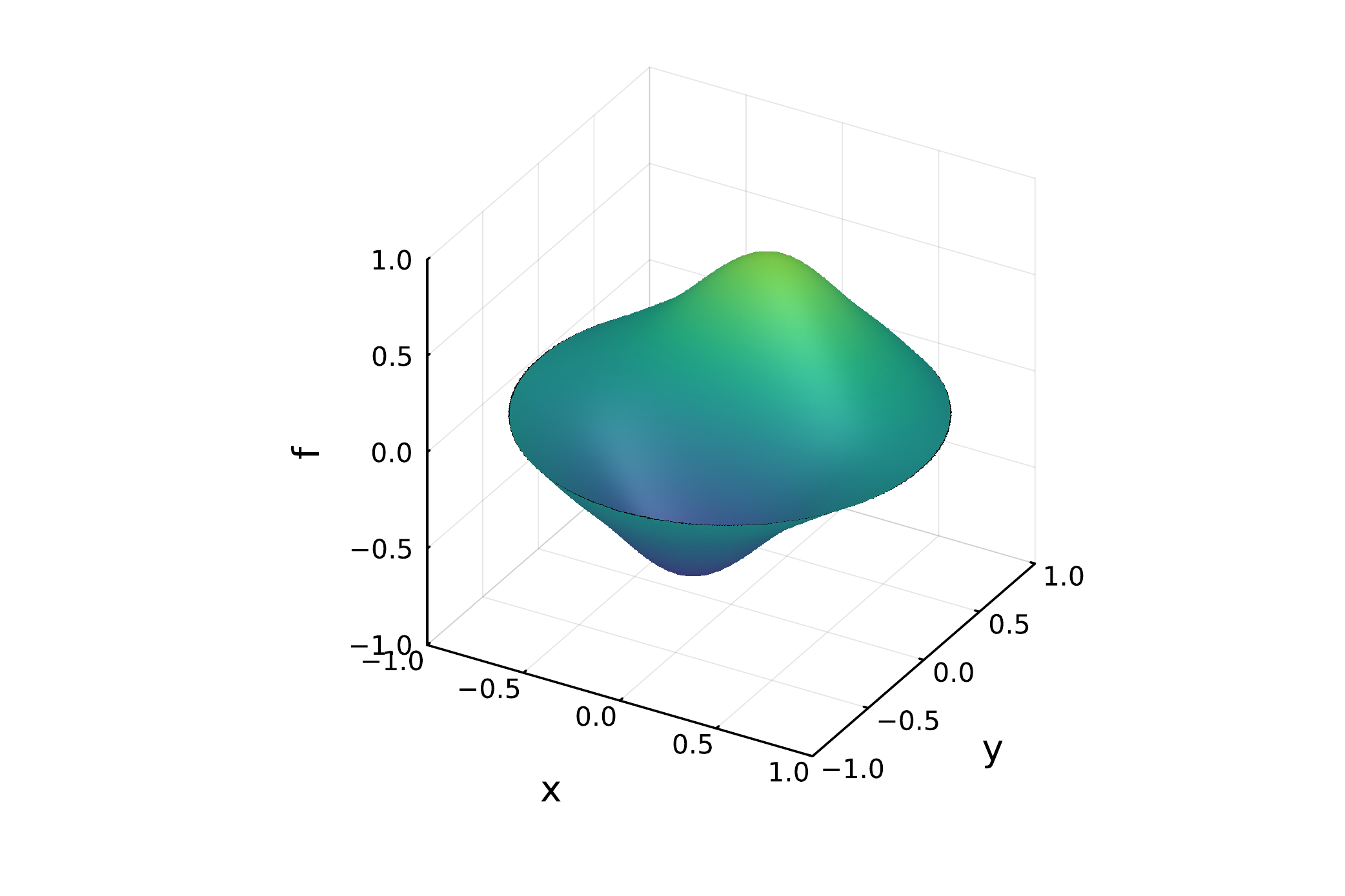} }}\\
         \subfloat[$T=38400$]
    {{ \centering \includegraphics[width=7cm]{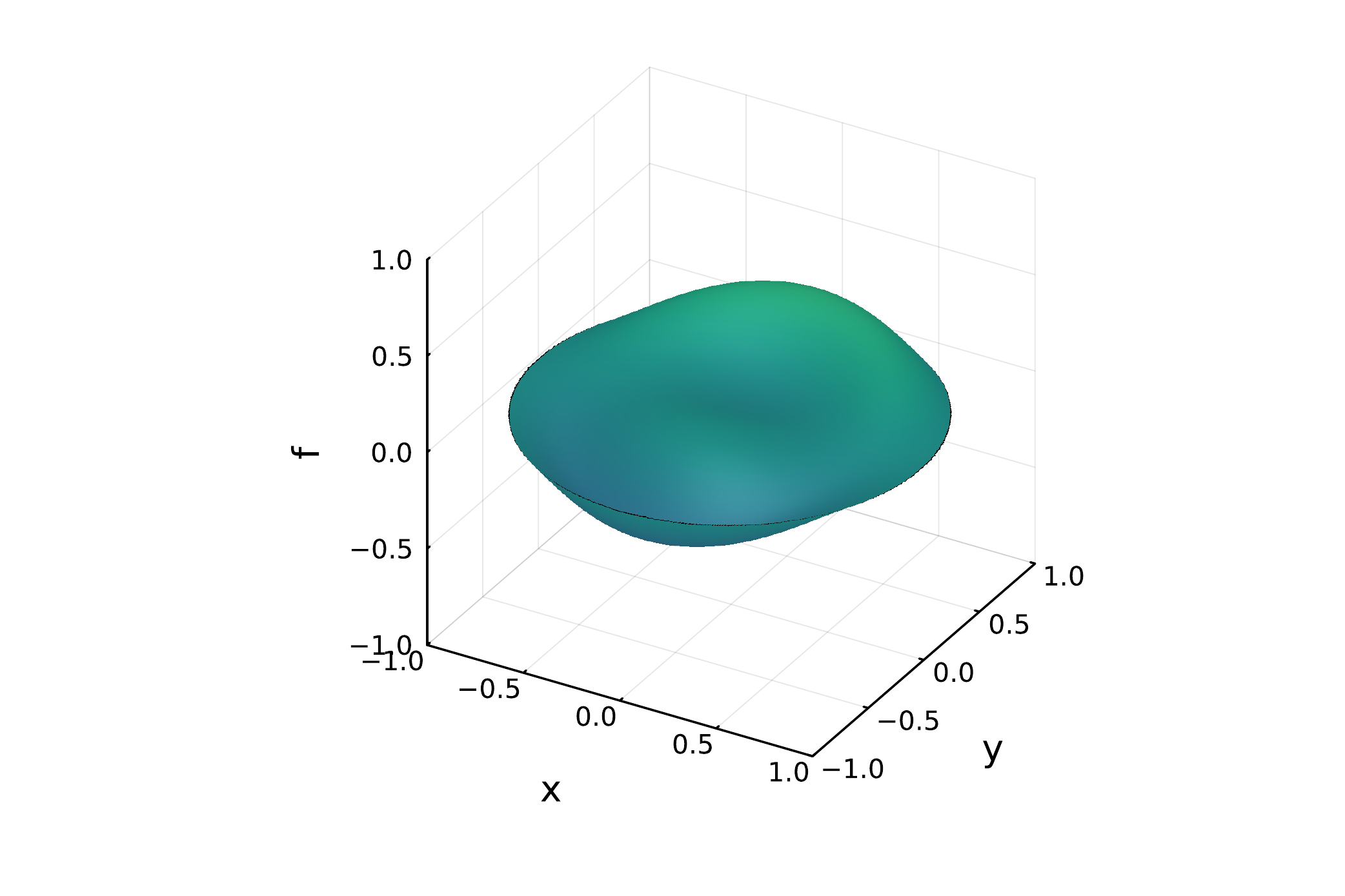} }}
     \subfloat[$T = 57600$]
    {{ \centering \includegraphics[width=7cm]{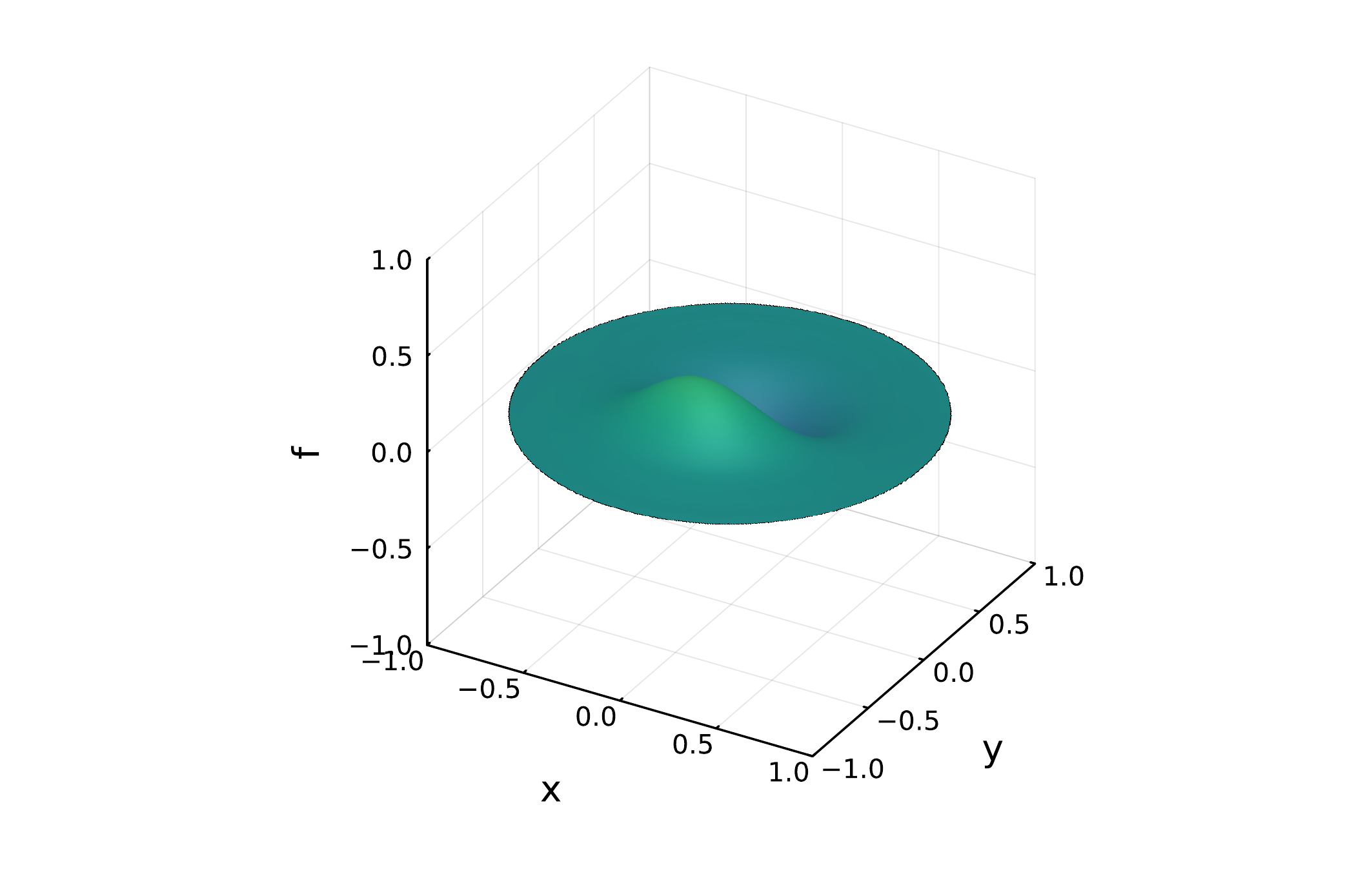} }}
    \caption{(b), (c) and (d) show numerical solutions to the fractionally dampened wave equation in Equation \eqref{eq:standinproblem} after indicated amount of time steps starting with the initial conditions in \eqref{eq:diskinitialnumexp} pictured in (a). The numerical parameters for the pictured simulation were $(\Delta t, K, L) = (2^{-20}, 500, 50)$. Animations for deeper time as well as other initial conditions are available on FigShare, see \cite{gutleb_wave_2023}.}
        \label{fig:diskexample1surface}
    \end{figure}
        \begin{figure}\centering
     \subfloat[]
    {{ \centering \includegraphics[width=7cm]{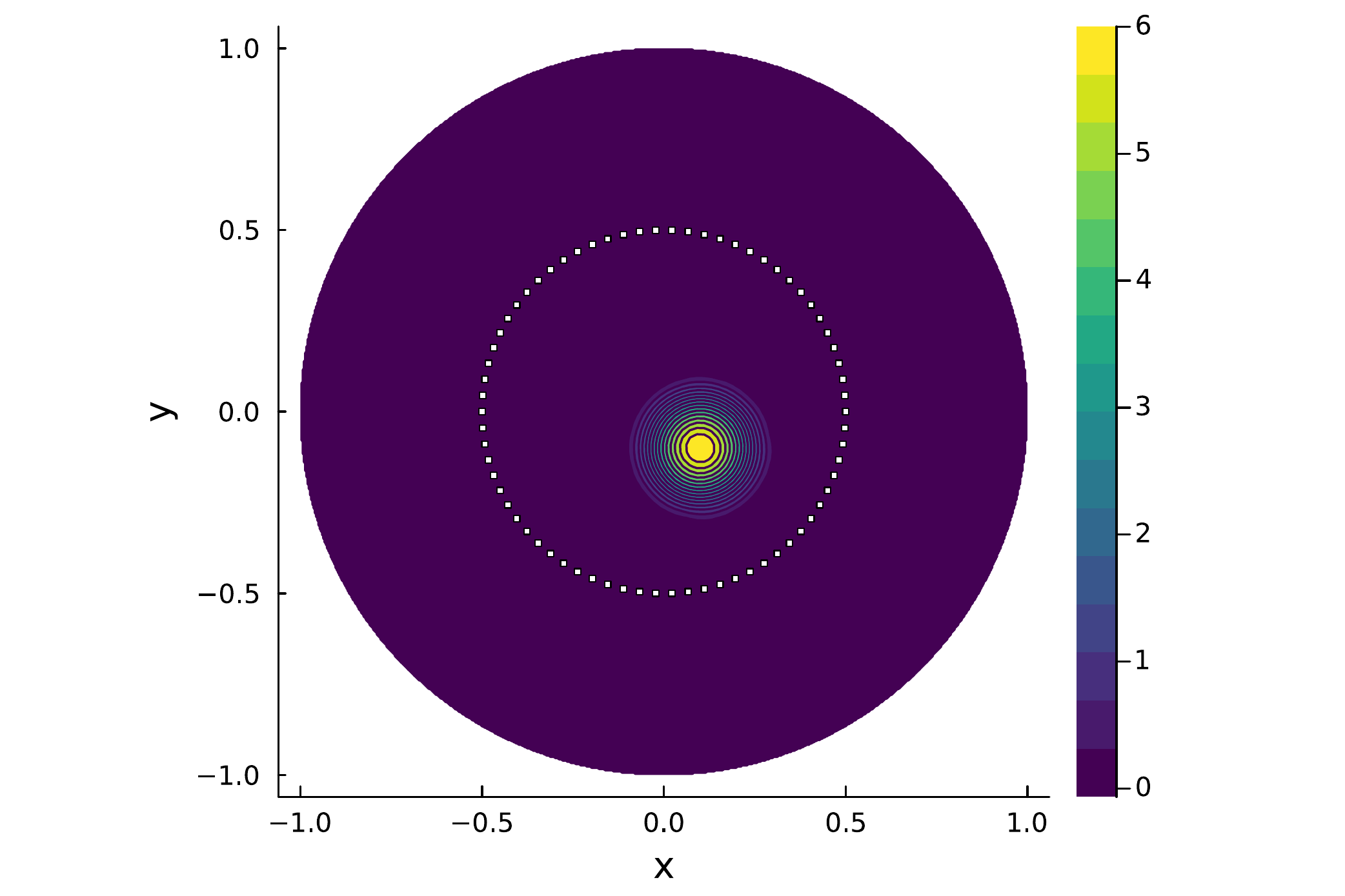} }}
     \subfloat[]
    {{ \centering \includegraphics[width=7cm]{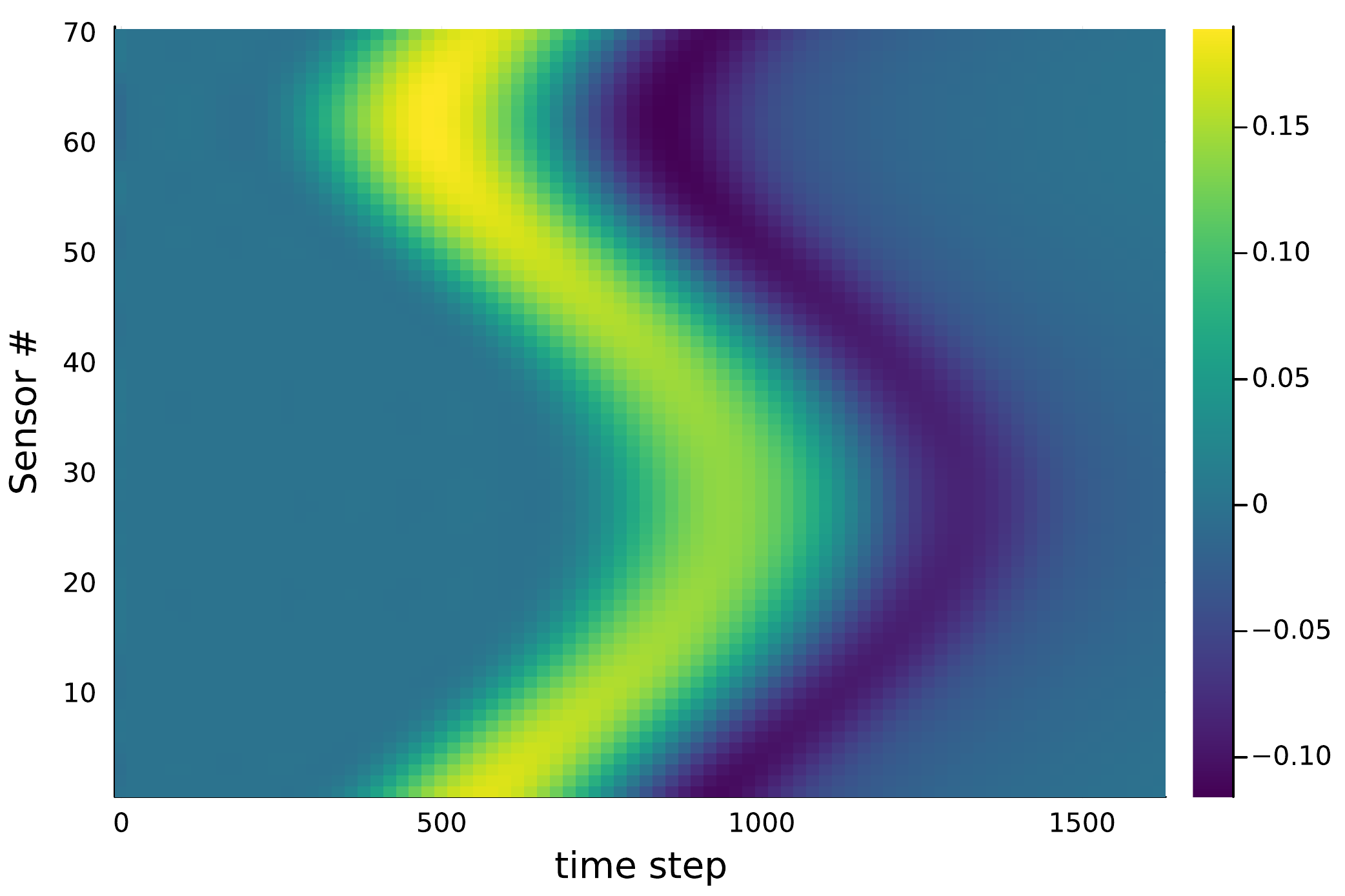} }}
    \caption{(b) shows the read-out as a function of time of the 70 sensors placed around the initial state plotted in (a) -- the circular sensor array with radius $r_s=\frac{1}{2}$ is marked by white squares in (a). The underlying equation is of the type in \eqref{eq:standinproblem} with $\alpha = \frac{1}{10}$. Inputs of the form of (b) can be used by the Matlab package k-wave to approximately reconstruct the original source images.}
        \label{fig:diskimagingexample}
    \end{figure}
\section{Discussion \& Future Research}\label{section:conclusion}
In this paper we have suggested a nonclassical sparse spectral approach to solving PDEs involving Caputo fractional derivatives. Among the method's core strengths are its recursive and history-free design, allowing computations that would otherwise be extremely challenging, and the reliance on sparse spectral methods in space resulting in banded operators which can be efficiently solved in each time step. The generic, almost `plug-and-play', nature of the approach means that it can easily be adapted to any domains on which sparse spectral methods are viable, i.e. all domains on which suitable sets of orthogonal polynomials and derivative operators are available. We demonstrated this by presenting numerical solutions to a time-fractional heat equation on the triangle and a time-fractional modification of the wave equation on the disk. As seen in the numerical experiments, the method does not require excessively many quadrature points to achieve good accuracy, with time step size $\Delta t$ becoming the primary error contributor once convergence in the number of quadrature points is achieved. This suggests that the approach outlined in this paper may in practice benefit significantly from replacing the various time stepping aspects of the method with higher order in time methods in a straightforward manner.\\
In future research we intend to adapt this method to work for the full equation of interest
\begin{align*}
\frac{1}{c_0^2} \frac{\partial^2}{\partial t^2}f - \Delta f - \tau \frac{\partial ^\alpha}{\partial t^\alpha} \Delta f = 0,
\end{align*}
which will require memory and computationally efficient extensions of this method for computing terms like $\frac{\partial^\alpha}{\partial t ^\alpha} (\Delta f(t,x))$ as well as a discussion of how to join up different domains with matching boundary conditions -- in this case combining a core ball cell with several layers of surrounding spherical shells of varying thickness. Applying such methods for the time reversed problem from finite sensor array readouts, as suggested with Figure \ref{fig:diskimagingexample} and the Matlab package k-wave, is of significant interest in applications. This functionality would require the use of theoretical results on the time-reversal of the discussed equations as well as an implementation of absorbing boundary conditions to prevent noise coming back from the domain boundary. To sensibly use absorbing boundary conditions with the present method, it is natural to e.g. use a spectral elements approach and enforce rapid decay in the outer-most element. We intend to discuss image reconstruction using a spectral element variant of this method in detail in a future paper.\\
While the above points remain a challenge for now, the introduced method's minimal theoretical overhead when changing domains suggest it is a good candidate for such multi-domain problems. Advances in this direction should be compared with the current state-of-the-art methods for solving these equations, which involve replacing the temporal nonlocality with a spatial nonlocality as discussed in \cite{treeby2014modeling} (as implemented in k-wave).

\section*{Acknowledgments}

TSG and JAC were supported by EPSRC grant number EP/T022132/1. JAC was supported by the Advanced Grant Nonlocal-CPD (Nonlocal PDEs for Complex Particle Dynamics: Phase Transitions, Patterns and Synchronization) of the European Research Council Executive Agency (ERC) under the European Union’s Horizon 2020 research and innovation programme (grant agreement No. 883363). JAC was also partially supported by the EPSRC grant number EP/V051121/1. The authors wish to thank Ioannis P. A. Papadopoulos and Sheehan Olver for providing helpful feedback on an early draft of the paper and Bradley Treeby for helpful conversations about medical ultrasound, k-wave's functionality and fractionally dampened wave equations.

 \bibliographystyle{elsarticle-num} 
 \bibliography{cas-refs}





\end{document}